\documentclass[11pt]{article}
\usepackage{color}

\usepackage{amssymb}   
\usepackage{amsthm}    
\usepackage{amsmath}   
\usepackage{stmaryrd}  
\usepackage{titletoc}  
\usepackage{mathrsfs}  
\usepackage{graphicx}

\newlength{\defbaselineskip}
\newcommand{\setlinespacing}[1]%
           {\setlength{\baselineskip}{#1 \defbaselineskip}}

\theoremstyle{plain}
\newtheorem{thm}{Theorem}[section]
\newtheorem{cor}[thm]{Corollary}
\newtheorem{lem}[thm]{Lemma}
\newtheorem{prop}[thm]{Proposition}

\theoremstyle{definition}
\newtheorem{defn}{Definition}[section]

\newtheorem{rmk}{Remark}[section]

\newcommand{\eps}{\varepsilon}

\DeclareMathOperator*{\esssup}{esssup}
\DeclareMathOperator*{\essinf}{essinf}

\newcommand{\cO}{\mathcal{O}}

\newcommand{\cL}{\mathcal{L}}
\newcommand{\cT}{\mathcal{T}}

\newcommand{\cB}{\mathcal{B}}
\newcommand{\cA}{\mathcal{A}}
\newcommand{\cS}{\mathcal{S}}

\newcommand{\cG}{\mathcal{G}}

\newcommand{\cU}{\mathcal{U}}

\newcommand{\bH}{\mathbb{H}}
\newcommand{\bP}{\mathbb{P}}
\newcommand{\bR}{\mathbb{R}}
\newcommand{\bN}{\mathbb{N}}

\newcommand{\sF}{\mathscr{F}}
\newcommand{\sP}{\mathscr{P}}

\makeatletter\@addtoreset{equation}{section} \makeatother
 \allowdisplaybreaks
\begin{document}

\title{Viscosity Solutions of Stochastic Hamilton-Jacobi-Bellman Equations\footnotemark[1] 
}

\author{Jinniao Qiu\footnotemark[2]  }
\footnotetext[1]{This work was partially supported by the National Science and Engineering Research Council of Canada. }
\footnotetext[2]{Department of Mathematics \& Statistics, University of Calgary, 2500 University Drive NW, Calgary, AB T2N 1N4, Canada. \textit{E-mail}: \texttt{jinniao.qiu@ucalgary.ca}. }
%
%

\maketitle

\begin{abstract}
In this paper we study the fully nonlinear stochastic Hamilton-Jacobi-Bellman (HJB) equation for the optimal stochastic control problem of stochastic differential equations with random coefficients. The notion of viscosity solution is introduced, and we prove that the value function of the optimal stochastic control problem is the maximal viscosity solution of the associated stochastic HJB equation. For the superparabolic cases when the diffusion coefficients are deterministic functions of time, states and controls, the uniqueness is addressed as well.

\end{abstract}



\section{Introduction}
Let $(\Omega,\sF,\{\sF_t\}_{t\geq0},\bP)$ be a complete filtered probability space carrying an $m$-dimensional Wiener process $W=\{W_t:t\in[0,\infty)\}$ such that $\{\sF_t\}_{t\geq0}$ is the natural filtration generated by $W$ and augmented by all the
$\bP$-null sets in $\sF$. We denote by $\sP$ the $\sigma$-algebra of the predictable sets on $\Omega\times[0,T]$ associated with $\{\sF_t\}_{t\geq0}$, and for each $t\geq 0$, $E_{\sF_t}[\,\cdot\,]$ represents the conditional expectation with respect to $\sF_t$.

Consider the following optimal stochastic control problem
\begin{align}
\inf_{\theta\in\cU}E\left[\int_0^T\!\! f(s,X_s,\theta_s)\,ds +G(X_T) \right] \label{Control-probm}
\end{align}
subject to
\begin{equation}\label{state-proces-contrl}
\left\{
\begin{split}
&dX_t=\beta(t,X_t,\theta_t)dt   +\sigma(t,X_t,\theta_t)\,dW_t,\,\,
\,t\in[0,T]; \\
& X_0=x,
\end{split}
\right.
\end{equation}
 where $T\in (0,\infty)$ is a fixed deterministic terminal time. Let $U\subset \bR^n$ be a nonempty compact set and $\cU$ the set of all the $U$-valued and $\sF_t$-adapted processes.   The process $(X_t)_{t\in[0,T]}$ is the {\sl state process}. It is governed by the {\sl control} $\theta\in\cU$. We sometimes write $X^{r,x;\theta}_t$ for $0\leq r\leq t\leq T$ to indicate the dependence of the state process on the control $\theta$, the initial time $r$ and initial state $x\in \mathbb{R}^d$.
 
In this paper, we consider the non-Markovian case where the coefficients $\beta,\sigma,f$ and $G$ depend not only on time, space and control but also \textit{explicitly} on $\omega \in\Omega$ (see assumption $(\cA1)$).    The dynamic cost functional is defined by
\begin{align}
J(t,x;\theta)=E_{\sF_t}\left[\int_t^T\!\! f(s,X^{t,x;\theta}_s,\theta_s)\,ds +G(X^{t,x;\theta}_T) \right],\ \ t\in[0,T] \label{eq-cost-funct}
\end{align}
and the value function is given by
\begin{align}
V(t,x)=\essinf_{\theta\in\cU}J(t,x;\theta),\quad t\in[0,T].
\label{eq-value-func}
\end{align}
In the spirit of dynamic programming principle, Peng conjectured (see \cite{Peng_92}) that the value function $V$ satisfies the stochastic HJB equation of the following form:
\begin{equation}\label{SHJB}
  \left\{\begin{array}{l}
  \begin{split}
  -du(t,x)=\,& 
 \mathbb{H}(t,x,Du,D^2u,D\psi) 
 \,dt -\psi(t,x)\, dW_{t}, \quad
                     (t,x)\in Q:=[0,T)\times\bR^d;\\
    u(T,x)=\, &G(x), \quad x\in\bR^d,
    \end{split}
  \end{array}\right.
\end{equation}
with 
\begin{align*}
\mathbb{H}(t,x,p,A,B)
= \essinf_{v\in U} \bigg\{
\text{tr}\left(\frac{1}{2}\sigma \sigma'(t,x,v) A+\sigma(t,x,v) B\right)
       +\beta'(t,x,v)p +f(t,x,v)
                \bigg\} 
\end{align*}
for $(p,A,B)\in \bR^d\times\bR^{d\times d}\times \bR^{m\times d}$,  where both the random fields $u(t,x)$ and $\psi(t,x)$ are  unknown. Along this line, a specific fully nonlinear stochastic HJB equation was formulated by Englezos and Karatzas \cite{EnglezosKaratzas09} for the utility maximization with habit formation, and more applications are  referred to \cite{bender2016-first-order,cardaliaguet2015master,Hu_Ma_Yong02} among many others.

The stochastic HJB equations are a class of backward stochastic partial differential equations (BSPDEs). The study of linear BSPDEs  dates back to  about forty years ago (see\cite{Pardoux1979}). They arise in many applications of probability theory and stochastic processes, for
instance in the nonlinear filtering and stochastic control theory for processes with incomplete information, as an adjoint equation of the Duncan-Mortensen-Zakai filtering equation (see \cite{Hu_Ma_Yong02,Tang_98,Zhou_93}). The representation relationship between forward-backward stochastic differential equations and BSPDEs yields the stochastic Feynman-Kac formula (see \cite{Hu_Ma_Yong02}). In addition, as the obstacle problems of BSPDEs, the reflected BSPDE arises as the HJB equation for the optimal stopping problems (see \cite{QiuWei-RBSPDE-2013}).

The linear and semilinear BSPDEs have been extensively studied, we refer to \cite{DuQiuTang10,DuTangZhang-2013,Hu_Ma_Yong02,ma2012non,Tang-Wei-2013} among many others. For the weak solutions and associated local behavior analysis for general quasi-linear BSPDEs, see \cite{QiuTangMPBSPDE11}, and we refer to \cite{Horst-Qiu-Zhang-14} for BSPDEs with singular terminal conditions. In the recent work \cite{Qiu2014weak}, the author studied the weak solution in Sobolev spaces for a special class of the fully nonlinear stochastic HJB equations (with $\beta\equiv 0$ and $\sigma(t,x,v)\equiv v$). The existence and uniqueness of solution for general cases is claimed as  an open problem in Peng's plenary lecture of ICM 2010 (see \cite{peng2011backward}).

In this paper, we propose a notion of viscosity solution for fully nonlinear stochastic HJB equations. The value function $V$ is verified to be the maximal viscosity solution, and for the superparabolic cases when the diffusion coefficients $\sigma$ do not depend explicitly on $\omega\in\Omega$ (see (ii) of $(\cA 3)$), the uniqueness is proved as well.  

Heuristically, the concerned random fields like the first unknown variable $u$ and the value function $V$ may be confined to the stochastic differential equations (SDEs) of the form:
\begin{align}
u(t,x)=u(T,x)-\int_{t}^T\mathfrak{d}_{s} u(s,x)\,ds-\int_t^T\mathfrak{d}_{w}u(s,x)\,dW_s,\quad (t,x)\in[0,T]\times\bR^d. \label{SDE-u}
\end{align}
The Doob-Meyer decomposition theorem implies the uniqueness of the pair $(\mathfrak{d}_tu,\,\mathfrak{d}_{\omega}u)$ and thus makes sense of the linear operators $\mathfrak{d}_t$ and $\mathfrak{d}_{\omega}$ which actually coincide with the two differential operators introduced by Le$\tilde{\text{a}}$o, Ohashi and Simas in \cite[Theorem 4.3]{Leao-etal-2018}. By contrast, we have $\psi=\mathfrak{d}_{\omega}u$ and to solve \eqref{SHJB} with a pair $(u,\psi)$ is equivalent to find $u$ (of form \eqref{SDE-u}) satisfying 
\begin{equation}\label{SHJB-eqiv}
  \left\{\begin{array}{l}
  \begin{split}
  -\mathfrak{d}_tu(t,x)
- \mathbb{H}(t,x,Du(t,x),D^2u(t,x),D\mathfrak{d}_{\omega}u(t,x))&=0,  \quad
                     (t,x)\in Q;\\
    u(T,x)&= G(x), \quad x\in\bR^d.
    \end{split}
  \end{array}\right.
\end{equation}
The equivalence relation between \eqref{SHJB} and \eqref{SHJB-eqiv} provides the key to defining the viscosity solutions for stochastic HJB equations. 
The main challenge lies in the  nonanticipativity constraints on the unknown variables and the fact that all the involved coefficients herein are only measurable w.r.t. $\omega$ on the sample space $(\Omega,\sF)$. This challenge prevents us from defining the viscosity solutions in a point-wise manner. To overcome this difficulty, we use a class of random fields of form \eqref{SDE-u} having sufficient spacial regularity as test functions and at each point $(\tau,\xi)$ ($\tau$ may be stopping time and $\xi$ may be an $\bR^d$-valued $\sF_{\tau}$-measurable variable) the classes of test functions are also parameterized by $\Omega_{\tau}\in\sF_{\tau}$ (see $\underline{\mathcal{G}}u(\tau,\xi;\Omega_{\tau})$ and $\overline{\mathcal{G}}u(\tau,\xi;\Omega_{\tau})$ in Section \ref{sec:def}).

We refer to \cite{crandall2000lp,crandall1992user,juutinen2001definition,wang1992-II} among many others for the theory of (deterministic) viscosity solutions and \cite{buckdahn2015pathwise,buckdahn2007pathwise,LionsSouganidis1998b} for the stochastic viscosity solutions of (forward) SPDEs. Note that the (backward) stochastic HJB equations like  \eqref{SHJB} and the (forward) ones studied in  \cite{buckdahn2007pathwise,LionsSouganidis1998b} 
are essentially different, i.e., the noise term in the latter is exogenous, while in the former it comes from the martingale representation and is governed by the coefficients, and thus it is endogenous. 

When the coefficients $\beta,\sigma,f$ and $G$ are deterministic functions of time $t$, control $\theta$ and the paths of $X$ and $W$, the optimal stochastic control problem is beyond the classical Markovian framework and the value function can be characterized by a path-dependent PDE. We refer to \cite{ekren2014viscosity,ekren2016viscosity-1,peng2011note,tang2012path} for the theory of viscosity solutions of  such nonlinear path-dependent PDEs. In particular, in both \cite{ekren2014viscosity,ekren2016viscosity-1}, the authors applied the path-dependent viscosity solution theory to some classes of stochastic HJB equations like \eqref{SHJB} which, however, required all the coefficients to be continuous in $\omega\in  \Omega$ due to the involved pathwise analysis. We would stress that, in the present work, all the involved coefficients are only measurable w.r.t. $\omega\in \Omega$ and we even do not need to specify any topology on $\Omega$, which allows the general random variables to appear in the coefficients.

The rest of this paper is organized as follows. In Section 2, we introduce some notations
and give the standing assumptions on the coefficients. Section 3 is devoted to some regular properties of the value function and the dynamic programming principle. In Section 4, we define the viscosity solution and prove the existence. In Section 5 we verify that the value function is the maximal viscosity solution  and then the uniqueness of viscosity solution is addressed for the superparabolic cases. Finally, in Appendix A  we recall a measurable selection theorem and comment on how it is used in this work, and Appendix B gives the proof of Proposition \ref{prop-value-func}.


\section{Preliminaries}

Denote by $|\cdot|$ the norm in  Euclidean spaces.
 Define the parabolic distance in $\bR^{1+d}$ as follows:
$$\delta(X,Y):=\max\{ |t-s|^{1/2},|x-y| \},$$
 for $X:=(t,x)$ and $Y:=(s,y)\in \bR^{1+d}$. Denote by $Q^+_r(X)$ the hemisphere of radius $r>0$ and center $X:=(t,x)\in \bR^{1+d}$ with $x\in \bR^d$:
\begin{equation*}
  \begin{split}
    Q^+_r(X):=\,  [t,t+r^2)\times B_r(x), \quad
    B_r(x):=\, \{ y\in\bR^n:|y-x|<r \},
  \end{split}
\end{equation*}
and by $|Q^+_r(X)|$ the volume. Throughout this paper, we write $(s,y)\rightarrow (t^+,x)$, meaning that $s \downarrow t$ and $y\rightarrow x$.

Let $\mathbb B $ be a Banach space equipped with norm $\|\cdot\|_{\mathbb B }$. For each $t\in[0,T]$, denote by $L^0(\Omega,\sF_t;\mathbb B)$ the space of $\mathbb B$-valued $\sF_t$-measurable random variables. For $p\in[1,\infty]$, $\cS ^p ({\mathbb B })$ is the set of all the ${\mathbb B }$-valued,
 $\sP$-measurable continuous processes $\{\mathcal X_{t}\}_{t\in [0,T]}$ such
 that
{\small $$\|\mathcal X\|_{\cS ^p({\mathbb B })}:= \left\|\sup_{t\in [0,T]} \|\mathcal X_t\|_{\mathbb B }\right\|_{L^p(\Omega,\sF,\bP)}< \infty.$$
}
 Denote by $\mathcal{L}^p({\mathbb B })$ the totality of all  the ${\mathbb B }$-valued,
  $\sP$-measurable processes $\{\mathcal X_{t}\}_{t\in [0,T]}$ such
 that
 {\small
 $$
 \|\mathcal X\|_{\mathcal{L}^p({\mathbb B })}:=\left\| \bigg(\int_0^T \|\mathcal X_t\|_{\mathbb B }^p\,dt\bigg)^{1/p} \right\|_{L^p(\Omega,\sF,\bP)}< \infty.
 $$
 }
Obviously, $(\cS^p({\mathbb B }),\,\|\cdot\|_{\cS^p({\mathbb B })})$ and $(\mathcal{L}^p({\mathbb B }),\|\cdot\|_{\mathcal{L}^p({\mathbb B })})$
are Banach spaces.

For each $(k,q)\in \mathbb{N}_0\times [1,\infty]$ we define  the $k$-th Sobolev space $(H^{k,q},\|\cdot\|_{k,q})$ as usual, and for each domain $\cO\subset \bR^d$, denote by $C^k(\cO)$ the space of functions with the up to $k$-th order derivatives being bounded and continuous on $\cO$, $C^k_0(\cO) $ being the subspace of $C^k(\cO)$ vanishing on the boundary $\partial \cO$. When $k=0$, write $C_0(\cO) $ and $C(\cO)$ simply. Through this paper, we define 
$$
\cS^p(C_{loc}(\bR^d))=\cap_{N>0} \cS^{p}(C(B_N(0))),\quad \text{ for } p\in[1,\infty].
$$

By convention, we treat elements of spaces like $\cS^p(H^{k,q})$ and $\cL^p(H^{k,q})$ as functions rather than distributions or classes of equivalent functions, and if a function of such class admits a version with better properties, we always denote this version by itself. For example, if $u\in \cL^p(H^{k,q})$ and $u$ admits a version lying in $\mathcal{S}^p(H^{k,q})$, we always adopt the modification $u\in \cL^p(H^{k,q})\cap \cS^p(H^{k,q})$.
\medskip

Throughout this work, we use the following assumption.

\bigskip
   $({\mathcal A} 1)$ \it $G\in L^{\infty}(\Omega,\sF_T;H^{1,\infty})$. For the coefficients $g=f,\beta^i,\sigma^{ij}$ $(1\leq i \leq d,\,1\leq j\leq m)$, \\
(i) $g:~\Omega\times[0,T]\times\bR^d\times U\rightarrow\bR$  {is}
$\sP\otimes\cB(\bR^d)\otimes\cB(U)\text{-measurable}$;\\
(ii) for almost all $(\omega,t)\in\Omega\times [0,T]$, $g(t,x,v)$ is uniformly continuous on $\bR^d\times U$;\\
(iii) there exists $L>0$ such that 
\begin{align*}
\|G\|_{L^{\infty}(\Omega,\sF_T;H^{1,\infty})} + \sup_{v\in\cU} \|g(\cdot,\cdot,v)\|_{\cL^{\infty}(H^{1,\infty})} \leq L.
\end{align*}



\section{Some properties of the value function and dynamic programming principle}

We first recall some standard properties of the strong solutions for SDEs (see \cite[Theorems 6.3 \& 6.16 of Chapter 1]{yong-zhou1999stochastic}). 
\begin{lem}\label{lem-SDE}
Let $(\cA1)$ hold. Given $\theta\in\cU$, for the strong solution of SDE \eqref{state-proces-contrl}, there exists $K>0$  such that, for any $0\leq r \leq t\leq s \leq T$ and $\xi\in L^p(\Omega,\sF_r;\bR^d)$ with $p\in[1,\infty)$,\\[3pt]
(i)   the two processes $\left(X_s^{r,\xi;\theta}\right)_{t\leq s \leq T}$ and $\left(X^{t,X_t^{r,\xi;\theta};\theta}_s\right)_{t\leq s\leq T}$ are indistinguishable;\\[2pt]
(ii)  $E_{\sF_r}\max_{r\leq l \leq T} \left|X^{r,\xi;\theta}_l\right|^p \leq K \left(1+ |\xi|^p\right)$ a.s.;\\[2pt]
(iii) $E_{\sF_r}\left| X^{r,\xi;\theta}_s-X^{r,\xi;\theta}_t  \right|^p \leq K \left(1+  |\xi|^p\right) (s-t)^{p/2}$ a.s.;\\[2pt]
(iv) given another $\hat{\xi}\in L^p(\Omega,\sF_r;\bR^d)$, 
$$
E_{\sF_r}\max_{r\leq l \leq T} \left|X^{r,\xi;\theta}_l-X^{r,\hat\xi;\theta}_l\right|^p \leq K  |\xi-\hat\xi|^p
\quad \text{a.s.};
$$
(v) the constant $K$ depends only on $L$, $T$ and $p$.
\end{lem}

For the dynamic cost functional defined in \eqref{eq-cost-funct}, the following lemma is an immediate application of \cite[Theorem 4.7 and Lemma 6.5]{Peng-DPP-1997}.

\begin{lem}\label{lem-J}
Let $\theta\in\cU$. For any $0\leq t\leq T$ and any $\xi\in L^p(\Omega,\sF_t;\bR^d)$ with $p\in[1,\infty]$, we have 
$$
J(t,\xi;\theta)=E_{\sF_t}\left[\int_t^T\!\! f(s,X^{t,\xi;\theta}_s,\theta_s)\,ds +G(X^{t,\xi;\theta}_T) \right] \quad \text{ a.s.,}
$$
and 
$$
V(t,\xi)=\essinf_{v\in\cU} J(t,\xi;v),\quad \text{ a.s.}
$$
\end{lem}

Some regular properties of the value function $V$ are then given below.
\begin{prop}\label{prop-value-func}
Let $(\cA 1)$ hold.
\\[4pt]
(i) For each  $t\in[0,T]$ and $\xi\in L^0(\Omega,\sF_t;\bR^d)$, there exists $\bar{\theta}\in\cU$ such that
$$
E\left[ J(t,\xi;\bar{\theta})-V(t,\xi)\right]<\eps.
$$
(ii) For each $(\bar{\theta},x)\in\cU\times \bR^d$, $\left\{J(t,X_t^{0,x;\bar\theta};\bar{\theta})-V(t,X_t^{0,x;\bar\theta})\right\}_{t\in[0,T]}$ is a supermartingale, i.e.,  for any $0\leq t\leq \tilde{t}\leq T$,
\begin{align}
V(t,X_t^{0,x;\bar\theta})
\leq E_{\sF_t}V(\tilde{t},X_{\tilde{t}}^{0,x;\bar{\theta}}) + E_{\sF_t}\int_t^{\tilde{t}}f(s,X_s^{0,x;\bar{\theta}},\bar\theta_s)\,ds,\,\,\,\text{a.s.}\label{eq-vfunc-supM}
\end{align}
(iii) For each $(\bar{\theta},x)\in\cU\times \bR^d$, $\left\{V(s,X_s^{0,x;\bar{\theta}})\right\}_{s\in[0,T]}$ is a continuous process.
\\[3pt]
(iv) 
 There exists $L_V>0$ such that for any $\theta\in\cU$
$$
|V(t,x)-V(t,y)|+|J(t,x;\theta)-J(t,y;\theta)|\leq L_V|x-y|,\,\,\,\text{a.s.},\quad \forall\,x,y\in\bR^d,
$$
with $L_V$ depending only on $T$ and the uniform Lipschitz constants of the coefficients $\beta,\sigma,f$ and $G$ w.r.t. the spatial variable $x$.
\\[3pt]
(v) With probability 1, $V(t,x)$ and $J(t,x;\theta)$ for each $\theta\in\cU$ are continuous  on $[0,T]\times\bR^d$ and 
$$ \sup_{(t,x)\in[0,T]\times\bR^d}   \max\left\{|V(t,x)|,\,|J(t,x;\theta)| \right\}       \leq L(T+1) \quad\text{a.s.}$$
\end{prop}

The proof of Proposition \ref{prop-value-func} is more or less standard and it is put in Appendix \ref{appdx-proof}. We then turn to present a generalized dynamic programming principle.


\begin{thm}\label{thm-DPP}
Let assumption $(\cA1)$ hold. For any stopping times $\tau,\hat\tau $ with $\tau\leq \hat\tau \leq T$, and any $ \xi\in L^p(\Omega,\sF_{\tau};\bR^d)$ for some $p\in[1,\infty]$, we have
\begin{align*}
V(\tau,\xi)=\essinf_{\theta\in\cU} E_{\sF_{\tau}}
\left[
\int_{\tau}^{\hat\tau} f\left(s,X_s^{\tau,\xi;\theta},\theta_s\right)\,ds + V\left(\hat\tau,X^{\tau,\xi;\theta}_{\hat \tau}\right)
\right] \quad a.s.
\end{align*}
\end{thm}

A proof of the generalized dynamic programming principle can be found  in \cite[Theorem 6.6]{Peng-DPP-1997} where the coefficients are required to be uniformly $\alpha$-H$\ddot{o}$lder ($\alpha \in (0,1]$) continuous in the control $\theta$. For the reader's convenience, we provide a concise proof below, since we assume only the uniform continuity of the coefficients ($\beta$, $\sigma$ and $f$) in the control $\theta$. 

\begin{proof}[Proof of Theorem \ref{thm-DPP}]
Denote the right hand side by $\overline V(\tau,\xi)$. In view of the definition for $V(t,x)$, we have obviously $\overline V(\tau,\xi)\leq V(\tau,\xi)$. 

For each $\eps>0$, by Proposition \ref{prop-value-func} (iv), there exists $\delta=\eps/L_V>0$ such that whenever $|x-y|<\delta$,
$$
|J(\hat\tau,x;\theta)- J(\hat\tau,y;\theta)  | + 
|V(\hat\tau,x)-V(\hat\tau,y)|
\leq \eps \quad \text{a.s., }\forall \, \theta\in\cU.
$$
Let $\{D^{j}\}_{j\in\bN^+}$ be a Borel partition of $\bR^d$ with diameter diam$(D^j)<\delta$, i.e., $D^j\in \cB(\bR^d)$, $\cup_{j\in\bN^+}D^j=\bR^d$, $D^i\cap D^j=\emptyset$ if $i\neq j$, and for any $x,y\in D^j$, $|x-y|<\delta$. For each $j\in\bN^+$, choose $x^j\in D^j$ and there exists $\theta^j\in\cU$ such that 
$$0\leq J(\hat\tau,x^j;\theta^j) -V(\hat\tau,x^j):= \alpha^j\quad \text{a.s., with } E|\alpha^j|<\frac{\eps}{2^j}.$$
Thus, for any $x\in D^j$,
\begin{align*}
&J(\hat\tau,x;\theta^j) -V(\hat\tau,x)\\
&\leq 
	|J(\hat\tau,x;\theta^j) -J(\hat\tau,x^j;\theta^j)|
	+|J(\hat\tau,x^j;\theta^j)-V(\hat\tau,x^j)|
	+|V(\hat\tau,x^j) -V(\hat\tau,x)|
\\
&\leq 2\,\eps+\alpha^j,\quad \text{a.s.}
\end{align*}

For any $\theta\in\cU$, set
\[  \tilde{\theta}_s=
\begin{cases}
\theta_s, \quad & \text{if }s\in[0,\hat\tau);\\
\sum_{j\in\bN^+} \theta^j_s1_{D^j}(X^{\tau,\xi;\theta}_{\hat\tau}),\quad &\text{if }s\in [\hat\tau,T].\\
\end{cases}
\]
Then it follows that
\begin{align*}
V(\tau,\xi)
&\leq J(\tau,\xi;\tilde\theta)\\
&=	E_{\sF_{\tau}}\left[
	\int_{\tau}^{\hat \tau}  
		f\left(s,X_s^{\tau,\xi;\theta},\theta_s\right)\,ds
	+J\left(\hat\tau,X_{\hat\tau}^{\tau,\xi;\theta};\tilde\theta\right)
	\right]\\
&\leq
	E_{\sF_{\tau}}\left[
	\int_{\tau}^{\hat \tau}  
		f\left(s,X_s^{\tau,\xi;\theta},\theta_s\right)\,ds
		+V\left(\hat\tau,X_{\hat\tau}^{\tau,\xi;\theta}\right)
		+\sum_{j\in\bN^+}\alpha^j
	\right]+2\,\eps
\end{align*}
where $\{\alpha^j\}$ is independent of the choices of $\theta$. Taking infimums and then expectations on both sides, we arrive at
\begin{align*}
EV(\tau,\xi) \leq
E \overline V(\tau,\xi) + 3\,\eps.
\end{align*}
By the arbitrariness of $\eps>0$, we have $E\overline V(\tau,\xi)\geq  EV(\tau,\xi)$, which together with the obvious relation $\overline V(\tau,\xi)\leq V(\tau,\xi)$ yields
that $\overline V(\tau,\xi)=V(\tau,\xi)$ a.s. 
\end{proof}


\section{Existence of viscosity solutions for stochastic HJB equations}

\subsection{Definition of viscosity solutions}\label{sec:def}

\begin{defn}\label{defn-testfunc}
For $u\in \cS^{2} (C(\bR^d))\cap \cL^2(C^2(\bR^d))$, we say $u\in \mathscr C_{\sF}^2$ if 
there exists $(\mathfrak{d}_tu, \,\mathfrak{d}_{\omega}u)\in \cL^2(C(\bR^d))\times  \cL^2(C^1(\bR^d))$ such that with probability 1
\begin{align*}
u(r,x)=u(T,x)-\int_r^T \mathfrak{d}_su(s,x)\,ds -\int_r^T\mathfrak{d}_{\omega}u(s,x)\,dW_s,\quad \forall\,(r,x)\in[0,T]\times\bR^d.
\end{align*}
\end{defn}


The fact $u\in \mathscr C_{\sF}^2 $ indicates that $\{u(t,x)\}_{0\leq t \leq T}$ is an It\^o process and thus a semi-martingale for each $x\in\bR^d$. Doob-Meyer decomposition theorem implies the uniqueness of the pair $(\mathfrak{d}_tu,\,\mathfrak{d}_{\omega}u)$. In this sense, by Definition \ref{defn-testfunc}, not only is the space $\mathscr C_{\sF}^2$ characterized, but the two linear operators $\mathfrak{d}_t$ and $\mathfrak{d}_{\omega}$ are also defined. In fact, during the finalization of this work, we found that Le$\tilde{\text{a}}$o, Ohashi and Simas \cite{Leao-etal-2018} had just defined a kind of weak differentiability of square-integrable It\^o processes w.r.t. $W$, and  for each $u\in \mathscr C_{\sF}^2$ and $x\in\bR^d$, the process $\{u(t,x)\}_{t\in [0,T]}$  can be thought of as an It\^o process and a straightforward application of \cite[Theorem 4.3]{Leao-etal-2018} indicates that $\mathfrak{d}_t$ and $\mathfrak{d}_{\omega}$ coincide with the two differential operators w.r.t. the paths of Wiener process $W$ in the sense of \cite{Leao-etal-2018}, which are defined via  a finite-dimensional approximation procedure based on controlled inter-arrival times and approximating martingales.   In particular, if $u(t,x)$ is a deterministic time-space function, BSDE theory yields that $\mathfrak{d}_{\omega}u \equiv 0$ and $\mathfrak{d}_tu$ coincides with the classical partial derivative in time; if the random function $u(t,x)$ is regular enough, its existing Malliavin derivative is nothing but $\mathfrak{d}_{\omega}u$. 

The linear operators $\mathfrak{d}_t$ and $\mathfrak{d}_{\omega}$ can be extended onto different spaces. In fact, the space $ \mathscr C^2_{\sF}$ is defined in an analogous way to the stochastic Banach spaces $\mathscr H^k_2$ in \cite[Definition 4.1]{DuQiuTang10} where for each element $u$ it is required that $u(T,\cdot)\in H^{k-1,2}$, $\mathfrak{d}_tu\in \cL^2(H^{k-2,2})$ and $\mathfrak{d}_{\omega}u\in \cL^2(H^{k-1,2})$. We would also note that the operators $\mathfrak{d}_t$ and $\mathfrak{d}_{\omega}$ here are different from the path derivatives $(\partial_t,\,\partial_{\omega})$ via the functional It\^o formulas (see \cite{buckdahn2015pathwise} and \cite[Section 2.3]{ekren2016viscosity-1}). If $u(\omega,t,x)$ is smooth enough w.r.t. $(\omega,t)$ in the path space,  for each $x$,  we have the relation 
$$\mathfrak{d}_tu(\omega,t,x)=\left(\partial_t+\frac{1}{2}\partial^2_{\omega\omega}\right)u(\omega,t,x),\quad \mathfrak{d}_{\omega}u(\omega,t,x)=  \partial_{\omega}u(\omega,t,x),$$      which can be seen either from the applications in \cite[Section 6]{ekren2016viscosity-1} to BSPDEs or from a rough view on the pathwise viscosity solution of (forward) SPDEs in \cite{buckdahn2015pathwise}.


For each stopping time $t\leq T$, denote by $\mathcal{T}^t$ the set of stopping times $\tau$ valued in $[t,T]$ and by $\mathcal{T}^t_+$ the subset of $\mathcal{T}^t$ such that $\tau>t$ for any $\tau\in \mathcal{T}^t_+$. For each $\tau\in\mathcal T^0$ and $\Omega_{\tau}\in\sF_{\tau}$, we denote by $L^0(\Omega_{\tau},\sF_{\tau};\bR^d)$ the set of $\bR^d$-valued $\sF_{\tau}$-measurable functions. 

We now introduce the notion of viscosity solutions. For each $(u,\tau)\in \cS^{2}(C_{loc}(\bR^d))\times \mathcal T^0$, $\Omega_{\tau}\in\sF_{\tau}$ with $\mathbb P(\Omega_{\tau})>0$ and $\xi\in L^0(\Omega_{\tau},\sF_{\tau};\bR^d)$, we define
\begin{align*}
\underline{\mathcal{G}}u(\tau,\xi;\Omega_{\tau}):=\bigg\{
\phi\in\mathscr C^2_{\sF}:(\phi-u)(\tau,\xi)1_{\Omega_{\tau}}=0=\essinf_{\bar\tau\in\mathcal T^{\tau}} E_{\sF_{\tau}}\left[\inf_{y\in B_{\delta}(\xi)}
(\phi-u)(\bar\tau\wedge \hat{\tau},y)
\right]1_{\Omega_{\tau}}  \text{ a.s.}
&\\
\text{for some } (\delta,\hat\tau) \in (0,\infty)\times\mathcal T^{\tau}_+\,&
\bigg\},\\
\overline{\mathcal{G}}u(\tau,\xi;\Omega_{\tau}):=\bigg\{
\phi\in\mathscr C^2_{\sF}:(\phi-u)(\tau,\xi)1_{\Omega_{\tau}}=0=\esssup_{\bar\tau\in\mathcal T^{\tau}} E_{\sF_{\tau}}\left[\sup_{y\in B_{\delta}(\xi)}
(\phi-u)(\bar\tau\wedge \hat{\tau},y)
\right]1_{\Omega_{\tau}}  \text{ a.s.}
&\\
\text{for some } (\delta,\hat\tau) \in (0,\infty)\times\mathcal T^{\tau}_+\,&
\bigg\}.
\end{align*}
It is obvious that if $\underline{\mathcal{G}}u(\tau,\xi;\Omega_{\tau})$ or $\overline{\mathcal{G}}u(\tau,\xi;\Omega_{\tau})$ is nonempty, we must have $0\leq\tau <T$ on $\Omega_{\tau}$.  

Now it is at the stage to introduce the definition of viscosity solutions.

\begin{defn}\label{defn-viscosity}
We say $u\in \cS^2(C_{loc}(\bR^d))$ is a viscosity subsolution (resp. supersolution) of BSPDE \eqref{SHJB}, if $u(T,x)\leq (\text{ resp. }\geq) G(x)$ for all $x\in\bR^d$ a.s., and for any $\tau\in  \mathcal T^0$, $\Omega_{\tau}\in\sF_{\tau}$ with $\mathbb P(\Omega_{\tau})>0$ and $\xi\in L^0(\Omega_{\tau},\sF_{\tau};\bR^d)$ and any $\phi\in \underline{\cG}u(\tau,\xi;\Omega_{\tau})$ (resp. $\phi\in \overline{\cG}u(\tau,\xi;\Omega_{\tau})$), there holds
\begin{align}
&\text{ess}\liminf_{(s,x)\rightarrow (\tau^+,\xi)}
	E_{\sF_{\tau}} \left\{ -\mathfrak{d}_{s}\phi(s,x)-\bH(s,x,D\phi(s,x),D^2\phi(s,x),D\mathfrak{d}_{\omega}\phi (s,x)) \right\}  \leq\ \,0,
\label{defn-vis-sub}
\end{align}
for almost all $\omega\in\Omega_{\tau}$ (resp.
\begin{align}
&\text{ess}\limsup_{(s,x)\rightarrow (\tau^+,\xi)} 
	E_{\sF_{\tau}} 
		\left\{ -\mathfrak{d}_{s}\phi(s,x)-\bH(s,x,D\phi(s,x),D^2\phi(s,x),D\mathfrak{d}_{\omega}\phi (s,x)) \right\}  \geq\ \,0,
\label{defn-vis-sup}
\end{align}
for almost all $\omega\in\Omega_{\tau}$).

Equivalently, $u\in \cS^2(C_{loc}(\bR^d))$ is a viscosity subsolution (resp. supersolution) of BSPDE \eqref{SHJB}, if $u(T,x)\leq (\text{ resp. }\geq) G(x)$ for all $x\in\bR^d$ a.s. and for any $\tau\in  \mathcal T^0$, $\Omega_{\tau}\in\sF_{\tau}$ with $\mathbb P(\Omega_{\tau})>0$ and $\xi\in L^0(\Omega_{\tau},\sF_{\tau};\bR^d)$ and any $\phi\in \mathscr C^2_{\sF}$, whenever there exist $\eps>0$, $\tilde\delta>0$ and $\Omega_{\tau}'\subset\Omega_{\tau}$ such that $\Omega'_{\tau}\in\sF_{\tau}$, $\bP(\Omega'_{\tau})>0$ and
{\small
\begin{align*}
&
	\essinf_{(s,x)\in Q^+_{\tilde\delta}(\tau,\xi) \cap Q}
	E_{\sF_{\tau}} \left\{ -\mathfrak{d}_{s}\phi(s,x)-\bH(s,x,D\phi(s,x),D^2\phi(s,x),D\mathfrak{d}_{\omega}\phi (s,x)) \right\} \geq  \eps \text{ a.e. in }\Omega'_{\tau}
\\
(\text{resp. }
&\esssup_{(s,x)\in Q^+_{\tilde\delta}(\tau,\xi) \cap Q}
	E_{\sF_{\tau}} \left\{ -\mathfrak{d}_{s}\phi(s,x)-\bH(s,x,D\phi(s,x),D^2\phi(s,x),D\mathfrak{d}_{\omega}\phi (s,x)) \right\} \leq  -\eps \text{ a.e. in }\Omega'_{\tau}),
\end{align*}
}
then $\phi\notin \underline{\cG}u(\tau,\xi;\Omega_{\tau})$ (resp. $\phi\notin \overline{\cG}u(\tau,\xi;\Omega_{\tau})$).

The function $u$ is a viscosity solution of BSPDE \eqref{SHJB} if it is both a viscosity subsolution and a viscosity supersolution of \eqref{SHJB}.
\end{defn}

 The test function space $ \mathscr C_{\sF}^2$ is expected to include the classical solutions of BSPDEs (see \cite{DuQiuTang10,Tang-Wei-2013} for instance). However, it is typical that the classical solutions $u$ may not be  differentiable in the time variable $t$ and $(\mathfrak{d}_tu,\mathfrak{d}_{\omega}u)$ may not be time-continuous but just measurable in $t$, which is also reflected in Definition \ref{defn-testfunc}.  This nature motivates us to use essential limits in \eqref{defn-vis-sub} and \eqref{defn-vis-sup}.  
 
 \begin{rmk}
 In view of Definition \ref{defn-viscosity}, we see:
 \\
 (i) The viscosity property is not discussed $\omega$-wisely but defined for each $(t,x,\Omega_{t})\in [0,T)\times \bR^d\times \sF_t$ with $\Omega_t$ being a sample set with positive probability. The viscosity property of function $u$ at $(t,x,\Omega_t)$ can be determined by  its values on $\left( Q^+_{\delta}(t,x)  \cap Q \right)\times \Omega_{t}$ equipped with appropriate filtration for any small $\delta>0$. In this sense, the viscosity property is local.
\\
(ii) As usual, when $ \underline{\cG}u(\tau,\xi;\Omega_{\tau})$ (resp. $ \overline{\cG}u(\tau,\xi;\Omega_{\tau})$) is empty, $u$ is automatically satisfying the viscosity subsolution (resp. supersolution) property at $(\tau,\xi)$.
\\
(iii) As standard in the literature of viscosity solutions for deterministic PDEs, we can choose smaller sets of test functions.  This will make the verifications of existence of viscosity solutions easier but complicate the uniqueness arguments. 
\end{rmk}

\begin{rmk}\label{rmk-defn-vs}
In the classical Markovian case where all the involved coefficients in problem \eqref{Control-probm} are deterministic, BSPDE \eqref{SHJB} becomes a deterministic parabolic PDE. Refined equivalent definitions of viscosity solutions for deterministic parabolic PDEs can be found in \cite{crandall2000lp,juutinen2001definition}.  In fact, by reversing time and using deterministic test functions, one can check that the definition above is consistent with the usual one for the continuous viscosity solutions, the difference being that our test functions are richer.

When the coefficients $\beta,\sigma,f$ and $G$ are deterministic functions of time $t$, control $\theta$ and the paths of $X$ and $W$, the optimal stochastic control problem is beyond the classical Markovian framework. Nevertheless, if one thinks of the $X$ and $W$ as state processes valued in the path space, the value function is deterministic and it can be characterized by a path-dependent PDE on the infinite-dimensional path space. We refer to \cite{ekren2014viscosity,ekren2016viscosity-1,peng2011note,tang2012path} for the theory of viscosity solutions of  such nonlinear path-dependent PDEs.  In fact, the authors in \cite{ekren2014viscosity,ekren2016viscosity-1} applied the path-dependent viscosity solution theory to some classes of stochastic HJB equations like \eqref{SHJB} which, however, required all the coefficients to be continuous in $\omega\in \Omega$ so that both the test functions and attained viscosity solution can be discussed pointwisely, while in this work, all the involved coefficients are only measurable w.r.t. $\omega\in  \Omega$ without any specified topology on $\Omega$, which along with the $\sigma$-algebras on $\Omega$ motivates us to discuss the test functions and viscosity solutions for each nontrivial measurable set $\Omega_{\tau}$ instead of defining viscosity solutions in a pointwise manner, and this method, in the definitions of test functions spaces $ \underline{\cG}u(\tau,\xi;\Omega_{\tau})$ and $ \overline{\cG}u(\tau,\xi;\Omega_{\tau})$, allows us to avoid the usage of nonlinear expectations that is an important technique in  \cite{ekren2014viscosity,ekren2016viscosity-1} to characterize the test functions.

\end{rmk}

To simplify the notations and involved techniques, we consider only the bounded continuous viscosity solutions in this paper and postpone to a future work more remarks on the viscosity solutions.

\subsection{Existence of the viscosity solution}

We first apply an It\^o-Kunita formula by Kunita \cite[Pages 118-119]{kunita1981some} to the composition of random fields and our controlled stochastic differential equations. Throughout this work, we define for any $\phi\in\mathscr C^2_{\sF}$ and $v\in U$,
\begin{align*}
\mathscr L^{v}\phi(t,x)=
 \mathfrak{d}_t \phi (t,x)+\text{tr}\left(\frac{1}{2}\sigma \sigma'(t,x,v) D^2\phi(t,x)
     	+\sigma(t,x,v) D\mathfrak{d}_{\omega}\phi(t,x)\right)       +D\phi(t,x)\beta(t,x,v).
\end{align*}
\begin{lem}\label{lem-ito-wentzell}
  Let assumption $(\cA1)$ hold.
 Suppose  $u\in\mathscr C_{\sF}^2$.
   Then, for each $x\in\bR^d$ and $\theta\in\cU$, it holds almost surely that, for all $t\in[0,T]$,
    \begin{align*}
 & u(t,X^{0,x;\theta}_t) -u(0,x)\\
&= 
     	\!\int_0^t    \mathscr L^{\theta_s} u\left(s,X^{0,x;\theta}_s \right)    \,ds      +\int_0^t\left( \mathfrak{d}_{\omega}u(r,X^{0,x;\theta}_r)    
     + Du(r,X^{0,x;\theta}_r)\sigma(r,X^{0,x;\theta}_r,\theta_r)\right)\,dW_r.
   \end{align*}
\end{lem}

\begin{thm}\label{thm-existence}
 Let $(\cA1)$ hold. The value function $V$ defined by \eqref{eq-value-func} is a viscosity solution of the stochastic Hamilton-Jacobi-Bellman equation \eqref{SHJB}.
\end{thm}

\begin{proof}
\textbf{Step 1.}
First, in view of Proposition \ref{prop-value-func}, we have $V\in\cS^{\infty}(C_{loc}(\bR^d))$. For each $\phi\in \underline\cG V(\tau,\xi;\Omega_{\tau})$ with $\tau\in \mathcal T^0$, $\Omega_{\tau}\in\sF_{\tau}$, $\mathbb P(\Omega_{\tau})>0$ and $\xi\in L^0(\Omega_{\tau},\sF_{\tau};\bR^d)$, let $(\hat{\tau},\,B_{\delta}(\xi))$ be the pair corresponding to $\phi\in \underline\cG V(\tau,\xi;\Omega_{\tau})$. 

Suppose to the contrary that there exist $\eps, \tilde\delta>0$ and $\Omega'\in\sF_{\tau}$ such that $\Omega'\subset \Omega_{\tau}$, $\bP(\Omega')>0$ and 
{\small
\begin{align}
\essinf_{(s,x)\in Q^+_{\tilde{\delta}}(\tau,\xi) \cap Q} 
	E_{\sF_{\tau}}
		\{ -\mathfrak{d}_{s}\phi(s,x)-\bH(s,x,D\phi(s,x),D^2\phi(s,x),D\mathfrak{d}_{\omega}\phi (s,x))\}  \geq 2\,\eps,\quad \text{a.e. in }\Omega'. \label{eq-ex-1}
\end{align}
}
By assumption (ii) of $(\cA 1)$ and the measurable selection theorem (see Theorem \ref{thm-MS}), there exists $\bar\theta\in \cU$ such that for almost all $\omega\in\Omega_{\tau}$,
$$
- \mathscr L^{\bar\theta_s}\phi(s,\xi)  -f(s,\xi,\bar\theta_s)   \geq   -\mathfrak{d}_{s}\phi(s,\xi)-\bH(s,\xi,D\phi(s,\xi),D^2\phi(s,\xi),D\mathfrak{d}_{\omega}\phi (s,\xi)) -       \eps
$$
for almost all $s$ satisfying $\tau\leq s<T$. This together with \eqref{eq-ex-1} implies 
\begin{align*}
\essinf_{\tau\leq s<\left(\tau+\tilde\delta^2\right)\wedge T} 
	E_{\sF_{\tau}}
		\{ - \mathscr L^{\bar\theta_s}\phi(s,\xi)  -f(s,\xi,\bar\theta_s) \}  \geq \eps,\quad \text{a.e. in }\Omega'.
\end{align*}
W.l.o.g., assume $\tilde{\delta}<\delta<1$. Define
$\tilde \tau=\inf\{ s>\tau: X^{\tau,\xi;{\bar\theta}}_s \notin B_{\tilde\delta/2}(\xi) \}$. Then $ \tilde \tau>\tau $ and moreover, for any $h>0$
 \begin{align}
 E_{\sF_{\tau}}\left[1_{\{\tilde \tau <\tau+h\}}\right]
 &=
 	E_{\sF_{\tau}}\left[1_{\{ \max_{\tau\leq s\leq \tau +h} |X_s^{\tau,\xi;\bar\theta} -\xi|>\frac{\tilde\delta}{2} \}}\right]  
\nonumber\\
&
 \leq 
 	\frac{16}{{\tilde\delta}^4}  E_{\sF_{\tau}}  \max_{\tau\leq s\leq \tau +h} |X_s^{\tau,\xi;\bar\theta} -\xi|^4
\nonumber\\
&
\leq 
	\frac{16K}{{\tilde\delta}^4} (1+|\xi|^4) h^2 \quad \text{ a.s.},
	\label{est-tau}
 \end{align}
where $K$ is from Lemma \ref{lem-SDE} and does not depend on the control $\bar\theta$.

By the dynamic programming principle of Theorem \ref{thm-DPP} and the  It\^o-Kunita formula of Lemma \ref{lem-ito-wentzell}, we have for any small $0<h<\tilde\delta^2/4$ and almost all $\omega\in\Omega'$, 
\begin{align*}
0 
&\geq \frac{1}{h}
	E_{\sF_{\tau}}\left[ 
		  (\phi-V)(\tau,\xi) -(\phi-V)\left((\tau+h)\wedge \tilde\tau\wedge\hat\tau,X_{(\tau+h)\wedge \tilde\tau\wedge\hat\tau}^{\tau,\xi;\bar\theta}\right)   \right]\\
&\geq \frac{1}{h } 
	E_{\sF_{\tau}} \left[ 
		   \phi(\tau,\xi) - \phi\left((\tau+h)\wedge \tilde\tau\wedge\hat\tau,X_{(\tau+h)\wedge \tilde\tau\wedge\hat\tau}^{\tau,\xi;\bar\theta}\right) 
		   -\int_{\tau}^{(\tau+h)\wedge \tilde\tau\wedge\hat\tau} f(s,X_s^{\tau,\xi;\bar\theta},\bar\theta_s)\,ds    \right]
\\
&=\frac{1}{h} E_{\sF_{\tau}}
	\int_{\tau}^{(\tau+h)\wedge \tilde\tau\wedge\hat\tau}
		   \left( -\mathscr L^{\bar\theta_s}  \phi(s,X_s^{\tau,\xi;\bar\theta})		- f(s,X_s^{\tau,\xi;\bar\theta},\bar\theta_s)   \right)\,ds
 \\
 &\geq
 	\frac{1}{h} E_{\sF_{\tau}} \bigg[
		\int_{\tau}^{(\tau+h)\wedge T}
		   \left( -\mathscr L^{\bar\theta_{s}}  \phi(s,X_{s\wedge \tilde\tau}^{\tau,\xi;\bar\theta})		- f(s,X_{s\wedge \tilde\tau}^{\tau,\xi;\bar\theta},\bar\theta_{s})   \right)\,ds
\\
&    \quad\quad\quad\quad
		   -  2\cdot 1_{\{\tau+h>\hat\tau\}\cup \{\tau+h>\tilde\tau\}   }
		   	\int_{\tau}^{T}
		  	 \left|   -\mathscr L^{\bar\theta_s}  \phi(s,X_{s\wedge \tilde\tau}^{\tau,\xi;\bar\theta})		- f(s,X_{s\wedge \tilde\tau}^{\tau,\xi;\bar\theta},\bar\theta_{s})   \right|\,ds
			 \bigg]
\\
&\geq
 	\frac{1}{h} E_{\sF_{\tau}} \bigg[
		\int_{\tau}^{(\tau+h)\wedge T}
		   \left( -\mathscr L^{\bar\theta_{s}}  \phi(s,\xi)		- f(s,\xi,\bar\theta_{s})   \right)\,ds
\\
&\quad\quad
		-\int_{\tau}^{ (\tau+h)\wedge T}
		   \left| -\mathscr L^{\bar\theta_{s}}  \phi(s,\xi)		- f(s,\xi,\bar\theta_{s})  +  \mathscr L^{\bar\theta_{s}}  \phi(s,X_{s\wedge \tilde\tau}^{\tau,\xi;\bar\theta})
		   +  f(s,X_{s\wedge \tilde\tau}^{\tau,\xi;\bar\theta},\bar\theta_{s})
		   \right|\,ds
\\
&    \quad\quad\quad\quad
		   -  2\cdot 1_{\{\tau+h>\hat\tau\}\cup \{\tau+h>\tilde\tau\}   }
		   	\int_{\tau}^{T}
		  	 \left|   -\mathscr L^{\bar\theta_s}  \phi(s,X_{s\wedge \tilde\tau}^{\tau,\xi;\bar\theta})		- f(s,X_{s\wedge \tilde\tau}^{\tau,\xi;\bar\theta},\bar\theta_{s})   \right|\,ds
			 \bigg]
\\
&\geq
	\frac{(T\wedge(\tau+h))-\tau}{h}\cdot \eps-o(1),
\end{align*}
where the term $o(1)$ that tends to zero as $h\rightarrow 0^+$ consists of three facts:
(i) the spatial regularity of $\phi\in\mathscr C^2_{\sF}$ indicates that as $h\rightarrow 0^+$, for almost all $\omega\in\Omega'$
$$
\frac{1}{h}\int_{\tau}^{ (\tau+h)\wedge T}
		   \left| -\mathscr L^{\bar\theta_{s}}  \phi(s,\xi)		- f(s,\xi,\bar\theta_{s})  +  \mathscr L^{\bar\theta_{s}}  \phi(s,X_{s\wedge \tilde\tau}^{\tau,\xi;\bar\theta})
		   +  f(s,X_{s\wedge \tilde\tau}^{\tau,\xi;\bar\theta},\bar\theta_{s})
		   \right|\,ds
		   \rightarrow 0;
$$
(ii) for almost all $\omega\in\Omega'$, $\tau +h <\hat\tau$ when $h>0$ is small enough;
(iii) estimate \eqref{est-tau}.
This incurs a contradiction as $h$ tends to zero and thus, for almost all $\omega\in\Omega_{\tau}$ 
\begin{align*}
&\text{ess}\liminf_{(s,x)\rightarrow (\tau^+,\xi)}
	E_{\sF_{\tau}} \left\{ -\mathfrak{d}_{s}\phi(s,x)-\bH(s,x,D\phi(s,x),D^2\phi(s,x),D\mathfrak{d}_{\omega}\phi (s,x)) \right\}  \leq\ \,0.
\end{align*}
Hence, $V$ is a viscosity subsolution of BSPDE \eqref{SHJB}.\medskip

\textbf{Step 2.}
It remains to prove that $V$ is a viscosity supersolution of \eqref{SHJB}. Let $\phi\in \overline\cG V(\tau,\xi;\Omega_{\tau})$ with $\tau\in \mathcal T^0$, $\Omega_{\tau}\in\sF_{\tau}$, $\mathbb P(\Omega_{\tau})>0$ and $\xi\in L^0(\Omega_{\tau},\sF_{\tau};\bR^d)$. Let $(\hat{\tau},\,B_{\delta}(\xi))$ be the pair corresponding to $\phi\in \overline\cG V(\tau,\xi;\Omega_{\tau})$. 


We argue with contradiction like in \textbf{Step 1}, To the contrary, assume that there exist $\eps, \tilde\delta>0$ and $\Omega'\in\sF_{\tau}$ such that $\Omega'\subset \Omega_{\tau}$, $\bP(\Omega')>0$ and 
\begin{align*}
\esssup_{(s,x)\in Q^+_{\tilde{\delta}}(\tau,\xi) \cap Q} 
	E_{\sF_{\tau}}
		\{ -\mathfrak{d}_{s}\phi(s,x)-\bH(s,x,D\phi(s,x),D^2\phi(s,x),D\mathfrak{d}_{\omega}\phi (s,x))\}  \leq -\eps,\quad \text{a.e. in }\Omega'.
\end{align*}
W.l.o.g., assume $\tilde{\delta}<\delta<1$. For each $\theta\in\cU$, define
$\tau^{\theta}=\inf\{ s>\tau: X^{\tau,\xi;{\theta}}_s \notin B_{\tilde\delta/2}(\xi) \}$. Estimate \eqref{est-tau} still holds.

By Theorem \ref{thm-DPP} and Lemma \ref{lem-ito-wentzell}, we have for any small $h\in(0,\tilde\delta^2/4)$ and almost all $\omega\in\Omega'$, 
\begin{align*}
0
&\geq
	\frac{V(\tau,\xi)-\phi(\tau,\xi)}{h}
\\
&=
	\frac{1}{h} \essinf_{\theta\in\cU} E_{\sF_{\tau}}\left[
	\int_{\tau}^{\hat{\tau} \wedge (\tau+h)} 
		f(s,X_s^{\tau,\xi;\theta},\theta_s)\,ds
		+V\left(\hat{\tau} \wedge (\tau+h), X_{\hat{\tau}\wedge (\tau+h)}^{\tau,\xi;\theta}\right) -\phi(\tau,\xi)
		\right]
\\
&\geq
	\frac{1}{h} \essinf_{\theta\in\cU} E_{\sF_{\tau}}\left[
	\int_{\tau}^{\hat{\tau} \wedge (\tau+h)} 
		f(s,X_s^{\tau,\xi;\theta},\theta_s)\,ds
		+\phi\left(\hat{\tau} \wedge (\tau+h), X_{\hat{\tau}\wedge (\tau+h)}^{\tau,\xi;\theta}\right) -\phi(\tau,\xi)
		\right]
\\
&=
	\frac{1}{h} \essinf_{\theta\in\cU} E_{\sF_{\tau}}\left[
	\int_{\tau}^{\hat{\tau} \wedge (\tau+h)} 
		\left( \mathscr L^{\theta_s}\phi\left(s,X_s^{\tau,\xi;\theta}\right)+f(s,X_s^{\tau,\xi;\theta},\theta_s)\right)\,ds
		\right]
\\
&\geq
	\frac{1}{h} \essinf_{\theta\in\cU} E_{\sF_{\tau}}\bigg[
		\int_{\tau}^{\tau^{\theta}\wedge (\tau+h)\wedge \hat\tau} \bigg(
		\mathscr L^{\theta_s}\phi\left(s,X_s^{\tau,\xi;\theta} \right)+f(s,X_s^{\tau,\xi;\theta},\theta_s)
      		  \bigg) \,ds\\
&	\quad\quad\quad
		  -1_{\{\hat\tau >\tau^{\theta}\}\cap\{\tau+h>		\tau^{\theta}\}}   \int_{\tau}^{T} \left| 
	\mathscr L^{\theta_s}\phi\left(s,X_s^{\tau,\xi;\theta} \right)+f(s,X_s^{\tau,\xi;\theta},\theta_s)
      		  \right| \,ds
		\bigg]
\\
&\geq
	\frac{1}{h} \essinf_{\theta\in\cU} E_{\sF_{\tau}}\bigg[
		\int_{\tau}^{(\tau+h)\wedge T} \bigg(
		\mathscr L^{\theta_s}\phi\left(s,X_{s\wedge\tau^{\theta}}^{\tau,\xi;\theta} \right)+f(s,X_{s\wedge\tau^{\theta}}^{\tau,\xi;\theta},\theta_s)
      		  \bigg) \,ds\\
&	\quad
		  -2\left(1_{\{\hat\tau >\tau^{\theta}\}\cap\{\tau+h>\tau^{\theta}\}}  	+  1_{\{\tau+h>\hat\tau\}\cup \{\tau+h>\tau^{\theta}\}   } \right)	   \int_{\tau}^{T} \left| 
	\mathscr L^{\theta_s}\phi\left(s,X_s^{\tau,\xi;\theta} \right)+f(s,X_s^{\tau,\xi;\theta},\theta_s)
      		  \right| \,ds
		\bigg]
\\
& \geq
\frac{(T\wedge(\tau+h))-\tau}{h}\cdot \eps-o(1),
\end{align*}
where we note that for almost all $\omega\in\Omega'$, $\tau +h <\hat\tau$ when $h>0$ is small enough and this along with estimate \eqref{est-tau} and relation $\phi\in\mathscr C^2_{\sF}$ gives the term $o(1)$ that tends to zero as $h\rightarrow 0^+$. A contradiction occurs as $h$ tends to zero. Thus, for almost all $\omega\in\Omega_{\tau}$ 
\begin{align*}
&\text{ess}\limsup_{(s,x)\rightarrow (\tau^+,\xi)}
	E_{\sF_{\tau}} \left\{ -\mathfrak{d}_{s}\phi(s,x)-\bH(s,x,D\phi(s,x),D^2\phi(s,x),D\mathfrak{d}_{\omega}\phi (s,x)) \right\}  \geq\ \,0.
\end{align*}
Hence, $V$ is a viscosity supersolution of BSPDE \eqref{SHJB}. This completes the proof.

\end{proof}


\section{Uniqueness of the viscosity solution}

The uniqueness consists of two parts. In the first subsection, we prove that the value function is the maximal viscosity (sub)solution of the stochastic HJB  equation \eqref{SHJB} that is fully nonlinear and can be degenerate. In the second subsection, the uniqueness is addressed for the superparabolic cases when the controlled diffusion coefficient $\sigma$ does not depend explicitly on $\omega\in\Omega$.

\subsection{Maximal viscosity solution}

We need further the following assumption.\\[8pt]
 $({\mathcal A} 2)$ \it For each $\theta\in\cU$, there exists $q> 2+\frac{d}{2}$ such that $(G(\cdot),f(\cdot,\cdot,\theta) )  \in     L^{2}(\Omega,\sF_T;H^{q,2}) \times \cL^{2}(H^{q,2}) $, and  $g(\cdot,\cdot,\theta) \in\cL^{\infty}(H^{q,\infty}) $ for $g=\beta^i,\sigma^{ij}$ $(1\leq i \leq d,\,1\leq j\leq m)$.\medskip\\
   
By the theory of degenerate BSPDEs (see  \cite[Theorem 2.1]{DuTangZhang-2013} and \cite[Theorem 4.3]{Du_Zhang_DegSemilin2012}), we have
\begin{lem}\label{lem-dBSPDE}
Let assumptions $(\cA 1)$ and $(\cA 2)$ hold.  For each $\theta\in\cU$, (see  \cite[Theorem 2.1]{DuTangZhang-2013}) the following BSPDE:
{\small
\begin{equation}\label{SHJB-dBSPDE}
  \left\{\begin{array}{l}
  \begin{split}
  -du(t,x)=\,&\displaystyle 
  \bigg\{\text{tr}\left(\frac{1}{2}\sigma \sigma'(t,x,\theta_t) D^2 u(t,x)+\sigma(t,x,\theta_t) D\psi(t,x)\right)
       +Du(t,x)\beta(t,x,\theta_t) \\ &\displaystyle
          \quad\quad\quad+f(t,x,\theta_t)
                \bigg\} \,dt -\psi(t,x)\, dW_{t}, \quad
                     (t,x)\in Q;\\
    u(T,x)=\, &G(x), \quad x\in\bR^d,
    \end{split}
  \end{array}\right.
\end{equation}
}
admits a unique solution $(u,\psi)$ with 
$u\in \cS_w^2(H^{q,2})$
 \footnote{ $\cS_w^2([0,T];H^{q,2})$ 
 denotes the space of all $H^{q,2}$-valued and jointly measurable processes $(u_{t})_{t\in [0,T]}$ which are $\mathscr{F}$-adapted, a.s. weakly continuous with respect to $t$ on $[0,T]$ (i.e., for any $h\in(H^{q,2})^*$, the dual space of $H^{q,2}$, the mapping $t \mapsto h(u_t)$ is a.s. continuous on $[0,T]$), and
$${E} \left[\sup_{t\in[0,T]}\|u_t\|_{H^{q,2}}^2\right] < \infty.$$
}
 and $\psi^k+\sum_{i=1}^dD_{x^i}u\sigma^{ik}(\cdot,\cdot,\theta)\in \cL^2(H^{q,2})$ for $k=1,\dots,m$. As $q> 2+\frac{d}{2}$, (by Sobolev embedding theorem and  \cite[Theorem 4.3]{Du_Zhang_DegSemilin2012}), it holds a.s. that $u(t,x)=J(t,x;\theta)$ for all $(t,x)\in[0,T]\times\bR^d$. 
\end{lem}   

\begin{rmk}\label{rmk-lem-dBSPDE}
In fact, analogous to \cite[Theorem A.1 and Remark A.1]{Horst-Qiu-Zhang-14}, we have further $J(\cdot,\cdot;\theta)\in \cS^2(H^{q-1,2})$ and as $q>2+d/2$, Sobolev embedding theorem and Proposition \ref{prop-value-func}  yield that $J(\cdot,\cdot;\theta)\in \cL^2(C_0^2(\bR^d))\cap \cS^{\infty}(C_0^1(\bR^d))$ and $J(\cdot,\cdot;\theta)\in \mathscr{C}^2_{\sF}$. In particular, the nonnegative part $J^+(\cdot,\cdot;\theta)\in \cS^{\infty}(C_0(\bR^d))$ and thus $V^+\in\cS^{\infty}(C_0(\bR^d))$. Moreover, with probability 1
$$
-\mathscr L^{\theta_s}J(s,x;\theta)-f(s,x,\theta_s) =0,\quad \text{for all }x\in\bR^d  \text{ and almost all }s\in[0,T),
$$
and thus with probability 1
{\small
\begin{align}
\text{ess}\liminf_{(s,x)\rightarrow (t^+,y)}
	E_{\sF_{t}} \left\{  -\mathfrak{d}_{s}J(s,x;\theta)-\bH(s,x,D J(s,x;\theta),D^2J(s,x;\theta),D\mathfrak{d}_{\omega}J (s,x;\theta)) \right\} \geq 0 \label{property-sup}
\end{align}
}
for  all $(t,y)\in[0,T)\times\bR^d$.
\end{rmk}


\begin{thm}\label{thm-max}
Let $(\cA1)$ and $(\cA2)$ hold.   Let $u$ be a viscosity solution  of the stochastic HJB equation \eqref{SHJB} with $u^+\in \cS^{2}(C_0(\bR^d))$. It holds a.s. that $u(t,x)\leq V(t,x)$ for any $(t,x)\in[0,T]\times\bR^d$, where $V$ is the value function defined by \eqref{eq-value-func}.
\end{thm}
\begin{proof}
We argue by contradiction.  Suppose that with a positive probability, $u(t,x)>V(t,x)$  at some point $(t,x)\in [0,T)\times \bR^d$. Then we have some $\theta\in\cU$ such that $u(t,x)>J(t,x;\theta)$ with a positive probability. Furthermore, there exist constant $\kappa>0$, $\Omega_t\in\sF_t$ and $\xi_t\in L^0(\Omega_t,\sF_t;\bR^d)$ such that $\bP(\Omega_t)>0$ and
$$\alpha:=u(t,\xi_t)-J(t,\xi_t;\theta)=\max_{x\in\bR^d} \{  u(t,x)-J(t,x;\theta) \}\geq \kappa\text{ for almost all }\omega\in\Omega_t,$$ 
where the existence and measurablity  of $\xi_t$ are from the measurable selection (see Theorem \ref{thm-MS}) and the facts $u^+\in \cS^2(C_0(\bR^d))$ and $J(\cdot,\cdot;\theta) \in  \cS^2(C_0(\bR^d))$. W.l.o.g, we take $\Omega_t=\Omega$.

For each $s\in(t,T]$, choose an $\sF_s$-measurable variable $\xi_s$ such that  
$$\left( u(s,\xi_s)-J(s,\xi_s;\theta) \right)^+=\max_{x\in\bR^d}   \left(u(s,x)-J(s,x;\theta)\right)^+ .$$  Set
\begin{align*}
Y_s&=(u(s,\xi_s) -J(s,\xi_s;\theta))^++\frac{\alpha (s-t)}{2(T-t)};\\
Z_s&= \esssup_{\tau\in\cT^s} E_{\sF_s} [Y_{\tau}].
\end{align*}
As $u^+\in \cS^2(C_0(\bR^d))$ and $J(\cdot,\cdot;\theta) \in  \cS^2(C_0(\bR^d))$, it follows obviously the time-continuity of  $\max_{x\in\bR^d}   \left(u(s,x)-J(s,x;\theta)\right)^+$ and thus that of
$\left( u(s,\xi_s)-J(s,\xi_s;\theta) \right)^+$. Therefore, the process $(Y_s)_{t\leq s \leq T}$ has continuous trajectories. 
Define $\tau=\inf\{s\geq t:\, Y_s=Z_s\}$. In view of the optimal stopping theory, observe that
$$
E_{\sF_t}Y_T=\frac{\alpha}{2}
	<\alpha=Y_t\leq Z_t=E_{\sF_t}Y_{\tau} =E_{\sF_t}Z_{\tau}.
$$
It follows that $\bP(\tau<T)>0$. As 
$$
(u(\tau,\xi_{\tau}) -J(\tau,\xi_{\tau};\theta))^++\frac{\alpha (\tau-t)}{2(T-t)}
=Z_{\tau}\geq E_{\sF_{\tau}}[Y_T]= \frac{\alpha}{2},
$$
we have $\bP((u(\tau,\xi_{\tau}) -J(\tau,\xi_{\tau};\theta))^+>0)>0$. Define $\hat\tau=\inf\{s\geq\tau:\, (u(s,\xi_{s}) -J(s,\xi_{s};\theta))^+\leq 0\}$. Obviously, $\tau\leq\hat\tau\leq T$.
 Put $\Omega_{\tau}=\{\tau<\hat\tau\}$. Then $\Omega_{\tau}\in \sF_{\tau}$ and $\bP(\Omega_{\tau})>0$.

Set $\phi(s,x)=J(s,x;\theta)-\frac{\alpha (s-t)}{2(T-t)}+E_{\sF_s}Y_{\tau}$. Then $\phi\in\mathscr{C}^2_{\sF}$ since $J(\cdot,\cdot;\theta)\in \mathscr{C}^2_{\sF}$. For each $\bar\tau\in\cT^{\tau}$, we have for almost all $\omega\in\Omega_{\tau}$, 
\begin{align*}
\left( \phi-u\right)(\tau,\xi_{\tau})
=0=Y_{\tau}-Z_{\tau}
\leq 
Y_{\tau}- E_{\sF_{\tau}}\left[Y_{\bar\tau\wedge \hat{\tau}} \right] 
= E_{\sF_{\tau}} \left[ \inf_{y\in\bR^d} (\phi-u)(\bar\tau\wedge\hat\tau,y)   \right],
\end{align*}
which together with the arbitrariness of $\bar\tau$ implies that $\phi\in \underline{\cG} u(\tau,\xi_{\tau};\Omega_{\tau})$. As $u$ is a viscosity subsolution, by property \eqref{property-sup} it holds that for almost all $\omega\in\Omega_{\tau}$, 
\begin{align*}
0
&\geq
	 \text{ess}\liminf_{(s,x)\rightarrow (\tau^+,\xi_{\tau})}
	E_{\sF_{\tau}} \left\{ -\mathfrak{d}_{s}\phi(s,x)-\bH(s,x,D\phi(s,x),D^2\phi(s,x),D\mathfrak{d}_{\omega}\phi (s,x)) \right\}
\\
&=
	\frac{\alpha}{2(T-t)}
\\
&\quad\quad
	+ \text{ess}\liminf_{(s,x)\rightarrow (\tau^+,\xi_{\tau})}
	E_{\sF_{\tau}} \left\{  -\mathfrak{d}_{s}J(s,x;\theta)-\bH(s,x,D J(s,x;\theta),D^2J(s,x;\theta),D\mathfrak{d}_{\omega}J (s,x;\theta)) \right\}
\\
&\geq
	\frac{\kappa}{2(T-t)}>0.	
\end{align*}
This is an obvious contradiction. 
\end{proof}

In the above proof, we adopt some similar techniques as in \cite[Proposition 5.3]{ekren2016viscosity-1} for the construction of stopping times $\tau$ and $\hat\tau$. Throughout the proof, we see that only the viscosity subsolution property of $u$ and the property \eqref{property-sup} of $J(\cdot,\cdot;{\alpha})\in\mathscr C^2_{\sF}$ with $\left(u(\cdot,\cdot)-J(\cdot,\cdot;\alpha)\right)^+\in  \cS^{2}(C_0(\bR^d))$ are used. Hence, omitting the proofs we have the following two corollaries.
\begin{cor}
Let $(\cA1)$ and $(\cA2)$ hold and $u$ be a viscosity subsolution of BSPDE \eqref{SHJB} with $u^+\in\cS^2(C_0(\bR^d))$. It holds a.s. that $u(t,x)\leq V(t,x)$ for any $(t,x)\in[0,T]\times\bR^d$, with $V$ being the value function defined by \eqref{eq-value-func}.
\end{cor}


\begin{cor}\label{cor-cmp}
Let $(\cA1)$ hold and $u$ be a viscosity subsolution (resp. supersolution)  of BSPDE \eqref{SHJB}  and $\phi\in\mathscr C^2_{\sF}$ with $(u-\phi)^{+}\in \cS^2(C_0(\bR^d)) $ (resp. $(\phi-u)^{+}\in \cS^2(C_0(\bR^d)) $), $\phi(T,x)\geq (\text{resp. }\leq) G(x)$ for all $x\in\bR^d$ a.s. and with probability 1
{\small
\begin{align*}
\text{ess}\liminf_{(s,x)\rightarrow (t^+,y)}
	E_{\sF_{t}} \left\{  -\mathfrak{d}_{s}\phi(s,x)-\bH(s,x,D \phi(s,x),D^2\phi(s,x),D\mathfrak{d}_{\omega}\phi (s,x)) \right\} \geq 0
\\
\text{(resp. }
\text{ess}\limsup_{(s,x)\rightarrow (t^+,y)}
	E_{\sF_{t}} \left\{  -\mathfrak{d}_{s}\phi(s,x)-\bH(s,x,D \phi(s,x),D^2\phi(s,x),D\mathfrak{d}_{\omega}\phi (s,x)) \right\} \leq 0\text{)}
\end{align*}
}
for  all $(t,y)\in[0,T)\times\bR^d$. It holds a.s. that $u(t,x)\leq$ (resp., $\geq$) $\phi(t,x)$,  $\forall \, (t,x)\in[0,T]\times\bR^d$.
\end{cor}

\subsection{Uniqueness of the viscosity solution for the superparabolic case}
We shall study the superparabolic cases. Rewrite the Wiener process $W=(\tilde{W},\bar{W})$ with $\tilde{W}$ and $\bar{W}$ being two mutually independent and respectively, $m_0$ and $m_1$($=m-m_0$) dimensional  Wiener processes. Here and in the following, we adopt the decomposition $\sigma=(\tilde{\sigma},\bar{\sigma})$ with $\tilde{\sigma}$ and $\bar{\sigma}$ valued in $\bR^{d\times m_0}$ and $\bR^{d\times m_1}$ respectively for the controlled diffusion coefficient $\sigma$, and associated with $(\tilde{W}, \bar{W})$. 
Denote by $\{\tilde{\sF}_t\}_{t\geq0}$ the natural filtration generated by $\tilde{W}$ and augmented by all the
$\bP$-null sets. 

\bigskip\medskip
   $({\mathcal A} 3)$ (i) \it For each $(t,x,{v})\in [0,T]\times\bR^d\times U$,  $G(x)$ is $\tilde{\sF}_T$-measurable and for the  random variables $h=\beta^i(t,x,v),f(t,x,v)$, $i=1,\cdot, d$,
$$
  h:~\Omega\rightarrow \bR
\text{ is } \tilde{\sF}_t \text{-measurable.}
$$
 (ii) The \textit{diffusion} coefficient $\sigma=(\tilde \sigma,\bar\sigma):~[0,T]\times\bR^d\times U\rightarrow\bR^{d\times m}$ is continuous and  does not depend on $\omega$, and there exists $\lambda\in(0,\infty)$  such that
   \begin{align*}
     \text{(Superparabolicity)}\quad \quad \sum_{i,j=1}^d\sum_{k=1}^{m_1} \bar\sigma^{ik}\bar\sigma^{jk}(t,x,v)\xi^i\xi^j\geq \lambda |\xi|^2\quad \,\,\forall\, (t,x,v,\xi)\in [0,T]\times\bR^d \times U\times\bR^d.
   \end{align*}
(iii) $G$ and $f$ are nonnegative random functions.\\[4pt]

\begin{rmk}\label{rmk-A3}
As stated in \cite{Qiu-2014-Hormander}, the adaptedness to some subfiltration like $(\tilde\sF_t)_{t\geq 0}$ in (i) of $(\cA 3)$ is necessary to have the superparabolicity in (ii) of $(\cA 3)$.  As for the randomness, the diffusion coefficient $\sigma$ is assumed to be a deterministic function of time, space and control. This is basically because we have only the boundedness of $DV$ which allows the randomness of $\beta$; in other words, if we have sufficient estimate of $D^2V$, then $\sigma$ can be a random variable like $\beta$ (see the arguments for estimate \eqref{est-A4} below). In view of the follwing proof of Theorem \ref{thm-uniqueness},  the assumptions on $\sigma$ can be indeed relaxed to be of the forms $\sigma(\tilde W_{t_1\wedge t},\dots,\tilde W_{t_N\wedge t},t,x,v)$ for some $N\in\bN^+$ (like $\beta^N$ in Lemma \ref{lem-approx} below), but it can not bear the same randomness as $\beta$ in this paper. Finally, assumption (iii) of $(\cA 3)$ indicates that the value function $V$ is nonnegative and this together with $V^+\in\cS^{\infty}(C_0(\bR^d))$ (see Remark \ref{rmk-lem-dBSPDE}) implies $V\in \cS^{\infty}(C_0(\bR^d))$. We note that the nonnegativity of $G$ and $f$ can be replaced equivalently by being bounded from below by some functions in $\cS^{\infty}(C_0(\bR^d))$.
\end{rmk}

\begin{lem}\label{lem-approx}
 Let $(\cA1)-(\cA3)$ hold. For each $\eps>0$, there exist partition $0=t_0<t_1<\cdot<t_{N-1}<t_N=T$ for some $N>3$ and functions 
$$(G^N,f^N,\beta^N)\in C^{3}(\bR^{m_0\times N}\times\bR^d)\times C( U;C^3([0,T]\times\bR^{m_0\times N}\times\bR^d)) \times C(U;C^3([0,T]\times\bR^{m_0\times N}\times\bR^d))$$  
such that 
\begin{align*}
&
	G^{\eps}:=\esssup_{x\in\bR^d} \left|G^N(\tilde W_{t_1},\cdots,\tilde W_{t_N},x)-G(x) \right|,
\\
& 
	f^{\eps}_t :=\esssup_{(x,v)\in\bR^d\times U}
		\left|f^N(\tilde W_{t_1\wedge t},\cdots,\tilde W_{t_N\wedge t},t,x,v)-f(t,x,v)\right|,\quad \text{for } t\in  [0,T],
\\
&
 	\beta^{\eps}_t :=\esssup_{(x,v)\in\bR^d\times U}
		\left|\beta^N(\tilde W_{t_1\wedge t},\cdots,\tilde W_{t_N\wedge t},t,x,v)-\beta(t,x,v)\right|,\quad \text{for } t\in  [0,T],
\end{align*}
 are $\tilde\sF_t$-adapted with
\begin{align*}
	\left\| G^{\eps}  \right\|_{L^2(\Omega,\tilde\sF_T;\bR)} + \left\| f^{\eps}  \right\|_{\cL^2(\bR)}   + \left\| \beta^{\eps}  \right\|_{\cL^2(\bR^d)}    <\eps,
\end{align*}
and $G^N$, $f^N$ and $\beta^N$ are uniformly Lipschitz-continuous in the space variable $x$ with an identical Lipschitz-constant $L_c$ independent of $N$ and $\eps$.
\end{lem}
Notice that the set of functions like $G^N(\tilde W_{t_1},\cdots,\tilde W_{t_N},x)$ (resp. $f^N(\tilde W_{t_1\wedge t},\cdots,\tilde W_{t_N\wedge t},t,x,v)$ and $\left(\beta^N\right)^i(\tilde W_{t_1\wedge t},\cdots,\tilde W_{t_N\wedge t},t,x,v)$) are dense in $L^{2}(\Omega,\sF_T;H^{q,2})$ (resp. $\cL^{2}(H^{q,2}) $ and $\cL^{\infty}(H^{q,\infty})$ ) for $q>2+\frac{d}{2}$, $i=1,\dots,\,d$ for each fixed $v$.   The proof of Lemma \ref{lem-approx} is an application of standard density arguments and it is omitted.

We are now ready to present the uniqueness result for the superparabolic cases.

\begin{thm}\label{thm-uniqueness}
Let assumptions $(\cA1)-(\cA3)$ hold. The value function $V$ defined by \eqref{eq-value-func} is the unique viscosity solution in $\cS^{\infty}(C_0(\bR^d))$ to BSPDE \eqref{SHJB}. Moreover, $V(t,x)$ is $\tilde{\sF}_t$-measurable for each $(t,x)\in[0,T]\times\bR^d$.
\end{thm}

For the proof of the uniqueness, we will adopt  a strategy  inspired by a variation of Perron's method (see \cite{ekren2016viscosity-1}). Indeed, even though Corollary \ref{cor-cmp} is not a (partial) comparison principle like in \cite{ekren2016viscosity-1}, it is sufficient for us to proceed with analogous schemes for the proof. 

Define 
\begin{align*}
\overline{\mathscr V}=
\bigg\{
	\phi\in\mathscr C^2_{\sF}:&\, \phi^-\in \cS^{\infty}(C_0(\bR^d)),\,\phi(T,x)\geq G(x)\,\,\,\forall x\in\bR^d, \text{ a.s., and with probability 1,}\\
&
	\text{ess}\liminf_{(s,x)\rightarrow (t^+,y)}	E_{\sF_t}\left[ -\mathfrak{d}_{s}\phi(s,x)-\bH(s,x,D\phi(s,x),D^2\phi(s,x),D\mathfrak{d}_{\omega}\phi (s,x)) \right]
	\geq 0\\
&
	\quad \forall (t,y)\in[0,T)\times\bR^d
	\bigg\}\\
\underline{\mathscr V}=
\bigg\{
	\phi\in\mathscr C^2_{\sF}:&\, \phi^+\in \cS^{\infty}(C_0(\bR^d)),\,\phi(T,x)\leq G(x)\,\,\,\forall x\in\bR^d, \text{ a.s., and with probability 1,}\\
&
		\text{ess}\limsup_{(s,x)\rightarrow (t^+,y)}  E_{\sF_t}\left[ -\mathfrak{d}_{s}\phi(s,x)-\bH(s,x,D\phi(s,x),D^2\phi(s,x),D\mathfrak{d}_{\omega}\phi (s,x)) \right] 
		\leq 0\\
&
		\quad \forall (t,y)\in[0,T)\times\bR^d
	\bigg\},
\end{align*}
and set
\begin{align*}
\overline{u}=\essinf_{\phi\in \overline{\mathscr V}} \phi, \quad 
\underline{u}=\esssup_{\phi\in \underline{\mathscr V}} \phi.
\end{align*}
In view of Corollary \ref{cor-cmp}, for any viscosity solution $u\in \cS^{\infty}(C_0(\bR^d))$ we have $\underline u\leq u\leq \overline u$.  Therefore,  for the uniqueness of viscosity solution, it is sufficient to check $\underline u=V= \overline u$.

\begin{proof}[Proof of Theorem \ref{thm-uniqueness}]
By assumption (iii) of $(\cA 3)$, the value function $V$ is nonnegative and this together with $V^+\in\cS^{\infty}(C_0(\bR^d))$ (see Remark \ref{rmk-lem-dBSPDE}) implies $V\in \cS^{\infty}(C_0(\bR^d))$. Thus, Corollary \ref{cor-cmp} yields $\underline u\leq V\leq \overline u$.

For each fixed $\eps\in(0,1)$, choose $(G^{\eps},\,f^{\eps},\,\beta^{\eps})$ and $(G^N,f^N,\beta^N)$ as in Lemma \ref{lem-approx}. Recalling the standard theory of backward SDEs, let the $\tilde \sF_t$-adapted and predictable pair  $(Y^{\eps},Z^{\eps})\in \cS^2(\bR)\times \cL^2(\bR^{m_0})$  be the solution  of backward SDE
$$
Y_s^{\eps}=G^{\eps}+
	\int_s^T\left(f^{\eps}_t+K\beta^{\eps}_t\right)\,dt
		-\int_s^TZ^{\eps}_s\,d\tilde W_s,
$$
and for each $(s,x)\in[0,T)\times\bR^d$,  set
{\small
\begin{align*}
{V}^{\eps}(s,x)
&=\essinf_{\theta\in\cU} E_{\sF_s} \left[
	\int_s^Tf^N\left(\tilde W_{t_1\wedge t},\cdots,\tilde W_{t_N\wedge t},t,X^{s,x;\theta,N}_t,\theta_t\right)\,dt
		+G^N\left(\tilde W_{t_1},\cdots,\tilde W_{t_N},X^{s,x;\theta,N}_T\right)
		\right],
\\
\overline{V}^{\eps}(s,x)
&=
	V^{\eps}(s,x)+Y^{\eps}_s,\\
\underline{V}^{\eps}(s,x)
&=
	V^{\eps}(s,x)-Y^{\eps}_s,
\end{align*}
}
where the constant $K\geq 0$ is to be determined later and $X^{s,x;\theta,N}_t$ satisfies SDE
\begin{equation*}
\left\{
\begin{split}
&dX_t=\beta^N(t,X_t,\theta_t)dt   +\sigma(t,X_t,\theta_t)\,dW_t,\,\,
\,t\in[s,T]; \\
& X_s=x.
\end{split}
\right.
\end{equation*}

We have $V^{\eps},\overline{V}^{\eps},\underline{V}^{\eps} \in \mathscr C^2_{\sF}$ which can be derived backwardly.  By the viscosity solution theory of fully nonlinear parabolic PDEs (see \cite{crandall2000lp,krylov-1987,lions-1983,Peng-DPP-1997,wang1992-II}), when $s\in[t_{N-1},T)$, $V^{\eps}(s,x)=\tilde V^{\eps}(s,x, \tilde W_{t_1},\cdots,\tilde W_{t_{N-1}},\tilde{W}_s)$ with
{\small
\begin{align*}
&\tilde{V}^{\eps}(s,x, \tilde W_{t_1},\cdots,\tilde W_{t_{N-1}},y)\\
&=\essinf_{\theta\in\cU} E_{\sF_s, \tilde W_s=y} \left[
	\int_s^Tf^N\left( \tilde W_{t_1},\cdots,\tilde W_{t_{N-1}},\tilde W_{t_N\wedge t},t,X^{s,x;\theta,N}_t,\theta_t\right)\,dt
		+G^N\left(\tilde W_{t_1},\cdots,\tilde W_{t_N},X^{s,x;\theta,N}_T\right)
		\right]
\end{align*}
}
satisfying the superparabolic HJB equation of the following form
{\small
\begin{equation}\label{HJB-N}
  \left\{\begin{array}{l}
  \begin{split}
  -D_tu(t,x,y)=\,& 
  \essinf_{v\in U} \bigg\{   \frac{1}{2} \text{tr}\left(D^2_{yy}u(t,x,y)\right) +
	  \text{tr}\Big(  \frac{1}{2}\sigma \sigma'( \tilde W_{t_1},\cdots,\tilde W_{t_{N-1}},y,t,x,v) D^2_{xx}u(t,x,y)
	  \\
	  &+\tilde  \sigma  ( \tilde W_{t_1},\cdots,\tilde W_{t_{N-1}},y,t,x,v)D^2_{xy}u(t,x,y)
	  \Big)\\
	&
		  +D_xu(t,x,y)\beta^N( \tilde W_{t_1},\cdots,\tilde W_{t_{N-1}},y,t,x,v)\\
	  &+f^N(\tilde W_{t_1},\cdots,\tilde W_{t_{N-1}},y,t,x,v)\bigg\},
                     \quad (t,x,y)\in [t_{N-1},T)\times\bR^d\times\bR^{m_0};\\
    u(T,x,y)=\, &G^N( \tilde W_{t_1},\cdots,\tilde W_{t_{N-1}},y,x), \quad (x,y)\in\bR^d\times\bR^{m_0},
    \end{split}
  \end{array}\right.
\end{equation}
}
and thus the regularity theory of viscosity solutions gives 
$$\tilde{V}^{\eps}(\cdot,\cdot, \tilde W_{t_1},\cdots,\tilde W_{t_{N-1}},\cdot)\in 
L^{\infty}\left(\Omega,\tilde\sF_{t_{N-1}}; C^{1+\frac{\bar\alpha}{2},2+\bar\alpha}([t_{N-1},T)\times\bR^d) \cap  C([t_{N-1},T]\times\bR^d)\right) 
,
$$
for some $\bar\alpha \in (0,1)$, where the \textit{time-space} H\"older space $C^{1+\frac{\bar\alpha}{2},2+\bar\alpha}([0,T)\times\bR^d)$ is defined as usual. We can make similar arguments on time interval $[t_{N-2},t_{N-1})$ taking the obtained $V^{\eps}(t_{N-1},x)$ as the terminal value, and recursively on intervals $[t_{N-3},t_{N-2})$, $\dots$, $[0,t_{1})$.
Furthermore, applying the  It\^o-Kunita formula to $\tilde{V}^{\eps}(s,x, \tilde W_{t_1},\cdots,\tilde W_{t_{N-1}},y)$ on $[t_{N-1},T]$ yields that
 {\small
 \begin{equation}\label{SHJB-N}
  \left\{\begin{array}{l}
  \begin{split}
  -dV^{\eps}(t,x)=\,& 
  \essinf_{v\in U} \bigg\{
	  \text{tr}\Big(  \frac{1}{2}\sigma \sigma'( \tilde W_{t_1},\cdots,\tilde W_{t_{N-1}},\tilde W_t,t,x,v) D_{xx}^2V^{\eps}(t,x)
	  \\
	  &+\tilde  \sigma  ( \tilde W_{t_1},\cdots,\tilde W_{t_{N-1}},\tilde W_t,t,x,v)D_x D_{y}\tilde{V}^{\eps}(t,x, \tilde W_{t_1},\cdots,\tilde W_{t_{N-1}},\tilde W_t)
	  \Big)\\
  &
  		  +D_xV^{\eps}(t,x)  \beta^N( \tilde W_{t_1},\cdots,\tilde W_{t_{N-1}},\tilde W_t,t,x,v)
	  +f^N(\tilde W_{t_1},\cdots,\tilde W_{t_{N-1}},\tilde W_t,t,x,v)\bigg\}\,dt\\
&
		-D_{y}\tilde{V}^{\eps}(t,x, \tilde W_{t_1},\cdots,\tilde W_{t_{N-1}},\tilde W_t) \,d\tilde W_t,
                     \quad (t,x)\in [t_{N-1},T)\times\bR^d;\\
    V^{\eps}(T,x)=\, &G^N( \tilde W_{t_1},\cdots,\tilde W_{t_{N-1}},\tilde W_T,x), \quad x\in\bR^d.
    \end{split}
  \end{array}\right.
\end{equation}
}
It follows similarly on intervals $[t_{N-2},t_{N-1})$, $\dots$, $[0,t_{1})$, and finally we have $V^{\eps},\overline{V}^{\eps},\underline{V}^{\eps} \in \mathscr C^2_{\sF}$. In particular, $\mathfrak{d}_{\omega} V^{\eps}$ is also constructed recursively, for instance, on $[t_{N-1},T)$, for $i=1,\dots,m_0$,
\begin{align*}
	\left(\mathfrak{d}_{\omega} V^{\eps}\right)^i (t,x)
&=
	\left(D_{y}\tilde{V}^{\eps}\right)^i(t,x, \tilde W_{t_1},\cdots,\tilde W_{t_{N-1}},\tilde W_t),\quad  (t,x)\in [t_{N-1},T)\times \bR^d,
\\
	\left(\mathfrak{d}_{\omega} \overline V^{\eps} \right)^i(t,x)
&
	=
	\left(D_{y}\tilde{V}^{\eps}\right)^i(t,x, \tilde W_{t_1},\cdots,\tilde W_{t_{N-1}},\tilde W_t)+Z_t^{\eps},\quad  (t,x)\in [t_{N-1},T)\times \bR^d,
\\
	\left(\mathfrak{d}_{\omega} \underline V^{\eps}\right)^i (t,x)
&
	=
	\left(D_{y}\tilde{V}^{\eps}\right)^i(t,x, \tilde W_{t_1},\cdots,\tilde W_{t_{N-1}},\tilde W_t) -Z^{\eps}_t,\quad  (t,x)\in [t_{N-1},T)\times \bR^d,
\end{align*}
and for $i=m_0+1,\dots,m$,
\begin{align*}
&
	\left(\mathfrak{d}_{\omega} V^{\eps}\right)^i (t,x)=\left(\mathfrak{d}_{\omega} \overline V^{\eps} \right)^i(t,x)=\left(\mathfrak{d}_{\omega} \underline V^{\eps}\right)^i (t,x)
=0,\quad  (t,x)\in [t_{N-1},T)\times \bR^d.
\end{align*}

In view of the approximation in Lemma \ref{lem-approx} and with an analogy to the proof of (iv) in Proposition \ref{prop-value-func}, there exists $\tilde{L}>0$ such that 
$$
\max_{(t,x)\in[0,T]\times \bR^d}\left\{ |DV^{\eps}(t,x)|  + |D\overline V^{\eps}(t,x)| + |D\underline V^{\eps}(t,x)|   \right\}
\leq \tilde L,\,\,\,\text{a.s.}
$$
with $\tilde L$ being independent of $\eps$ and $N$. Set $K=\tilde L$. Then for $\overline V^{\eps}$ on $[t_{N-1},T)$, omitting the inputs for each functions we have
\begin{align}
&-\mathfrak{d}_{t}\overline V^{\eps}-\bH(D\overline V^{\eps},D^2\overline V^{\eps},D\mathfrak{d}_{\omega}\overline V^{\eps})
\nonumber\\
&=
	-\mathfrak{d}_{t}\overline V^{\eps}
	-\essinf_{v\in U} \bigg\{
	  \text{tr}\Big(  \frac{1}{2}\sigma \sigma' D^2\overline V^{\eps}
	  +  \sigma D\mathfrak{d}_{\omega}\overline V^{\eps}
	  \Big)
  		  +   D\overline V^{\eps}\beta^N
	  +f^N+f^{\eps}+ \tilde L  \beta^{\eps}\nonumber\\
&
	  \quad\quad
	  +   D\overline V^{\eps} \left(\beta-\beta^N\right)      -\beta^{\eps} \tilde L
	  +f-f^N-f^{\eps}
	   \bigg\}
\nonumber\\
&\geq
	-\mathfrak{d}_{t}\overline V^{\eps}
	-\essinf_{v\in U} \bigg\{
	  \text{tr}\Big(  \frac{1}{2}\sigma \sigma' D^2\overline V^{\eps}
	  +  \sigma D\mathfrak{d}_{\omega}\overline V^{\eps}
	  \Big)
  		  +    D\overline V^{\eps}\beta^N
	  +f^N+f^{\eps}+\beta^{\eps} \tilde L 
	   \bigg\}
	 \label{est-A4}\\
&=0,\nonumber
\end{align}
and it follows similarly on intervals $[t_{N-2},t_{N-1})$, $\dots$, $[0,t_1)$ that
$$
-\mathfrak{d}_{t}\overline V^{\eps}-\bH(D\overline V^{\eps},D^2\overline V^{\eps},D\mathfrak{d}_{\omega}\overline V^{\eps}) \geq 0,
$$
which together with the obvious facts  $\overline V^{\eps}(T)=G^{\eps}+G^N\geq G$ and $\left(\overline V^{\eps}\right)^- \equiv 0$  indicates that $\overline V^{\eps}\in \overline {\mathscr V}$. Analogously, $\underline V^{\eps}\in \underline {\mathscr V}$.

Now let us measure the distance between $\underline V^{\eps}$, $\overline V^{\eps}$ and $V$. By the theory of backward SDEs, we first have 
\begin{align*}
\|Y^{\eps}\|_{\cS^2(\bR)} + \|Z^{\eps}\|_{\cL^2(\bR^{m_0})}
&
	\leq \bar K \left(  \|G^{\eps}\|_{L^2(\Omega,\tilde\sF_T;\bR)} + \|f^{\eps}+\tilde L \beta^{\eps}\|_{\cL^2(\bR)}  \right)
\\
&
	\leq  C\eps
\end{align*}
with the constant $C$ independent of $N$ and $\eps$.  Fix some $(s,x)\in[0,T)\times \bR^d$. In view of the approximation in Lemma \ref{lem-approx}, using It\^o's formula, Burkholder-Davis-Gundy's inequality, and Gronwall's inequality, we have through standard computations that for any $\theta\in\cU$,
\begin{align*}
&
	E_{\sF_s} \left [  \sup_{s\leq t\leq T} \left|  X^{s,x;\theta,N}_t-X^{s,x;\theta}_t   \right| ^2  \right]
\\
&\leq
	\tilde KE_{\sF_s}\int_s^T\left| \beta^N\left(\tilde W_{t_1\wedge t},\cdots,\tilde W_{t_N\wedge t},t,X^{s,x;\theta,N}_t,\theta_t\right)-\beta\left(t,X^{s,x;\theta,N}_t,\theta_t\right) \right|^2\,dt
	\\
&\leq
\tilde KE_{\sF_s}\int_s^T   \left|\beta^{\eps}_t\right|^2\,dt,
\end{align*}
with $\tilde K$ being independent of $N$, $\eps$ and $\theta$. Then
\begin{align*}
&E\left| V^{\eps}(s,x)-V(s,x)\right|
\\
&	
	\leq
		E\esssup_{\theta\in\cU}
		E_{\sF_s}\bigg[ 
		\int_s^T\Big( f^{\eps}_t +\Big| f\left(t,X^{s,x;\theta,N}_t,\theta_t\right) -f\left(t,X^{s,x;\theta}_t,\theta_t\right)\Big| \Big)\,dt
\\
&\quad\quad\quad
		+G^{\eps} +\Big| G\left(X^{s,x;\theta,N}_T\right) -G\left(X^{s,x;\theta}_T \right)\Big|
		\bigg] 
\\
&
	\leq E|Y^{\eps}_s|
		+
		2L(T^{1/2}+1)\tilde K E \esssup_{\theta\in\cU}
		\left(
		E_{\sF_s}\int_s^T   \left|\beta^{\eps}_t\right|^2\,dt
		 \right)^{1/2}
\\
&
	\leq \left\| Y^{\eps}\right\|_{\cS^2(\bR)} +2L\tilde K (T^{1/2}+1) \left\|  \beta^{\eps} \right\|_{\cL^2(\bR^d)} 
\\
&
	\leq K_0 \eps,
\end{align*}
with the constant $K_0$ being independent of $N$, $\eps$ and $(s,x)$. Furthermore, in view of the definitions of $\overline V^{\eps}$ and $\underline V^{\eps}$, there exists some constant $K_1$ independent of $\eps$ and $N$ such that
\begin{align*}
E\left| \overline V^{\eps}(s,x)-V(s,x)\right|
+E\left| \underline V^{\eps}(s,x)-V(s,x)\right|
\leq K_1 \eps, \quad \forall \, (s,x)\in [0,T]\times\bR^d.
\end{align*}
The arbitrariness of $\eps$ together with the relation $\overline V^{\eps}\geq V \geq \underline V^{\eps}$ finally implies that $\underline u=V= \overline u$. Moreover, $V(t,x)$ is $\tilde{\sF}_t$-measurable for each $(t,x)\in[0,T]\times\bR^d$.
\end{proof}

\begin{rmk}\label{rmk-superparab}
   Through the above proof we see that only  $V\in \cS^2(C_0(\bR^d))$ and the approximations in Lemma \ref{lem-approx} are expected from $(\cA 2)$. Thus,  the assumption $(\cA2)$ may be relaxed.  Besides, in the above proof we in fact construct the regular approximations of $V$ and this along with  $\tilde{\sF}_t$-adaptedness of $V(t,x)$
 for each $(t,x)\in[0,T]\times\bR^d$ indicates that under assumptions $(\cA1)-(\cA3)$, the stochastic HJB equation can be equivalently written:
 \begin{equation}\label{SHJB-super}
  \left\{\begin{array}{l}
  \begin{split}
  -du(t,x)=\,& 
  \essinf_{v\in U} \bigg\{
	  \text{tr}\Big(  \frac{1}{2}\sigma \sigma'(t,x,v) D^2u(t,x)
	  +\tilde  \sigma  (t,x,v)D \psi(t,x)
	  \Big)
  		  \\
&
	\quad\quad\quad
		  +   Du(t,x) \beta(t,x,v) +f( t,x,v)\bigg\}\,dt-\psi(t,x) \,d\tilde W_t,
		\\
&
\quad\quad\quad		
                     \quad (t,x)\in [0,T)\times\bR^d;\\
 u(T,x)=\, &G( x), \quad x\in\bR^d.
    \end{split}
  \end{array}\right.
\end{equation}
\end{rmk}

%




\begin{appendix}
\section{Measurable selection theorem}
The following measurable selection theorem is referred to \cite{wagner1977survey}.
\begin{thm}\label{thm-MS}
Let $(\Lambda,\mathscr M)$ be a measurable space equipped with a nonnegative measure $\mu$ and let $(\cO,\cB(\cO))$ be a polish space. Suppose $F$ is a set-valued function from $\Lambda$ to $\cB(\cO)$ satisfying: (i) for $\mu$-a.e. $\lambda \in\Lambda$, $F(\lambda)$ is a closed nonempty subset of $\cO$; (ii) for any open set $O\subset \cO$, $\{\lambda:\, F(\lambda)\cap O\neq \emptyset\}\in \mathscr M$. Then there exists a  measurable function $f$: $(\Lambda,\mathscr M)\rightarrow (\cO,\cB(\cO))$ such that for $\mu$-a.e. $\lambda\in\Lambda$, $f(\lambda)\in F(\lambda)$. 
\end{thm}

In \textbf{Step 1} of the proof of Theorem \ref{thm-existence}, we take $\Lambda=\{(\omega,t): \omega\in \Omega_{\tau}   \text{ and }  \tau(\omega)\leq t<T\}$, $\mathscr M=\sP$, $\mu=\bP\otimes dt$ and $(\cO,\cB(\cO))=(U,\cB(U))$. Then by the continuity of the involved functions, 
\begin{align*}
F(\omega,s):=
\bigg\{v\in U:\,&
- \mathscr L^{v}\phi(s,\xi)  -f(s,\xi,v)   \geq   \\
&\esssup_{\tilde{v}\in U} \left(- \mathscr L^{\tilde v}\phi(s,\xi)  -f(s,\xi,\tilde v) \right)  -    \eps
\bigg\}
\end{align*}
satisfies the hypothesis in Theorem \ref{thm-MS} and $\bar\theta$ can be constructed in an obvious way.

In the proof of Theorem \ref{thm-max}, for each $s\in [t,T]$, take $(\Lambda,\mathscr M,\mu)=(\Omega_t,\sF_s\cap \Omega_t,\bP)$ and $(\cO,\cB(\cO))=(\bR^d,\cB(\bR^d))$.
Noticing that both $u^+$ and $J(\cdot,\cdot;\theta)$ are lying in $\cS^2(C_0(\bR^d))$, define
$$
F(\omega)=\left\{
\hat x\in\bR^d:\,
\left(u(s,\hat x)-J(s,\hat x;\theta)\right)^+
=\max_{x\in\bR^d} \left( u(s,x)-J(s,x;\theta)\right)^+
\right\}.
$$
Applying  Theorem \ref{thm-MS} directly gives the existence of $(\xi_s)_{s\in[t,T]}$.

\section{Proof of Proposition \ref{prop-value-func}}\label{appdx-proof}

\begin{rmk}\label{rmk-infn}
For each $(t,\bar \theta)\in[0,T]\times \cU$ and $\xi\in L^0(\Omega,\sF_t;\bR^d)$,  set
$$
\mathbb{J}(t,\xi;\bar\theta)=\left\{
J(t,\xi;\theta): J(t,\xi;\theta)\leq J(t,\xi;\bar\theta),\,\,\theta\in\cU
\right\}.
$$
Then $\mathbb{J}(t,\xi;\bar\theta)$ is nonempty and for any $J(t,\xi;\tilde\theta),J(t,\xi;\check\theta)\in \mathbb{J}(t,\xi;\bar\theta)$, putting 
$$
\gamma_s=\bar\theta_s1_{\{s\in[0,t)\}}+
\left(\tilde\theta_s 1_{\{J(t,\xi;\tilde\theta)\leq J(t,\xi;\check\theta)\}}
+  \check\theta_s 1_{\{J(t,\xi;\tilde\theta)> J(t,\xi;\check\theta)\}}
\right)1_{\{s\in[t,T]\}},
$$
 one has $\gamma\in\cU$ and 
$$
J(t,\xi;\tilde\theta)\wedge J(t,\xi;\check\theta)=
J(t,\xi;\gamma)\in \mathbb{J}(t,\xi;\bar\theta).
$$
Hence, by \cite[Theorem A.3]{Karatz-Shreve-1998MMF}, there exists  $\{\theta^n\}_{n\in\bN^+}\subset \cU$ such that
$J(t,\xi;\theta^n)$ converges decreasingly to $V(t,\xi)$ with probability 1.
\end{rmk}

\begin{proof}[Proof of Proposition \ref{prop-value-func}]
From Remark \ref{rmk-infn}, assertion (i) follows obviously. Again by Remark \ref{rmk-infn}, there exists $\{\theta^n\}_{n\in\bN^+}\subset \cU$ such that $J(\tilde t,X^{0,x;\bar\theta}_{\tilde t};\theta^n)$ converges decreasingly to $V(\tilde t,X^{0,x;\bar\theta}_{\tilde t})$ with probability 1. Therefore, we have
{\small
\begin{align*}
&E_{\sF_t}
V(\tilde{t},X_{\tilde{t}}^{0,x;\bar{\theta}})
+ E_{\sF_t}\int_t^{\tilde{t}}f(s,X_s^{0,x;\bar{\theta}},\bar{\theta}_s)\,ds
\\
=\,&
E_{\sF_t}
\lim_{n\rightarrow \infty} 
J(\tilde t,X^{0,x;\bar\theta}_{\tilde t};\theta^n)
+ E_{\sF_t}\int_t^{\tilde{t}}f(s,X_s^{0,x;\bar{\theta}},\bar{\theta}_s)\,ds
\\
=\,&
\lim_{n\rightarrow \infty} 
E_{\sF_t}J(\tilde t,X^{0,x;\bar\theta}_{\tilde t};\theta^n)
+ E_{\sF_t}\int_t^{\tilde{t}}f(s,X_s^{0,x;\bar{\theta}},\bar{\theta}_s)\,ds
\\
=\,&
\lim_{n\rightarrow \infty}
E_{\sF_t}
\left[
\int_{\tilde{t}}^T
f\left(s,X^{\tilde t,X_{\tilde{t}}^{0,x;\bar{\theta}};\theta^n}_s,\theta^n_s\right)\,ds+
\int_t^{\tilde{t}}f\left(s,X_s^{0,x;\bar{\theta}},\bar{\theta}_s\right)\,ds+G\left(X^{\tilde t,X_{\tilde{t}}^{0,x;\bar{\theta}};\theta^n}_T\right)
\right]\\
\geq\,&
\essinf_{\theta\in\cU}E_{\sF_t}
\left[
\int_{t}^T
f\left(s,X^{ t,X_{{t}}^{0,x;\bar{\theta}};\theta}_s,\theta_s\right)\,ds +G\left(X^{ t,X_{{t}}^{0,x;\bar{\theta}};\theta}_T\right)
\right]
\\
=\,&V(t,X^{0,x;\bar\theta}_t),\quad \text{a.s.,}
\end{align*}
}
which yields \eqref{eq-vfunc-supM} as well as assertion (ii).

 Then we have for $0\leq t\leq \tilde{t}\leq T$,
{\small
\begin{align}
&L(\tilde t -t) \nonumber  \\
&\geq E_{\sF_t}\int_t^{\tilde{t}}f(s,X_s^{0,x;\bar{\theta}},\bar{\theta}_s)\,ds
	 \nonumber \\
&\geq\,V(t,X_t^{0,x;\bar\theta})-E_{\sF_t}V(\tilde{t},X_{\tilde{t}}^{0,x;\bar{\theta}}) \nonumber \\
&=
	\essinf_{\theta\in\cU}
	E_{\sF_t}\left[
	\int_{t}^T
	f(s,X^{ t,X_{{t}}^{0,x;\bar{\theta}};\theta}_s,\theta_s)\,ds +G\left(X^{ t,X_{{t}}^{0,x;\bar{\theta}};\theta}_T\right)
	\right] \nonumber \\
&\quad-E_{\sF_t}\essinf_{\theta\in\cU}
	E_{\sF_{\tilde{t}}}\left[
	\int_{\tilde{t}}^T
	f(s,X^{\tilde t,X_{\tilde{t}}^{0,x;\bar{\theta}};\theta}_s,\theta_s)\,ds +G\left(X_T^{\tilde t,X_{\tilde{t}}^{0,x;\bar{\theta}};\theta}\right)
	\right]
	 \nonumber \\
&\geq
	\essinf_{\theta\in\cU}
	E_{\sF_t}\left[
	\int_{t}^T
	f(s,X^{ t,X_{{t}}^{0,x;\bar{\theta}};\theta}_s,\theta_s)\,ds +G\left(X^{ t,X_{{t}}^{0,x;\bar{\theta}};\theta}_T\right)\right] \nonumber \\
&\quad-\essinf_{\theta\in\cU}
	E_{\sF_t}\left[
	\int_{\tilde{t}}^T
	f(s,X^{\tilde t,X_{\tilde{t}}^{0,x;\bar{\theta}};\theta}_s,\theta_s)\,ds +G\left(X_T^{\tilde t,X_{\tilde{t}}^{0,x;\bar{\theta}};\theta}\right)
	\right]
	 \nonumber \\
&\geq
	\essinf_{\theta\in\cU}\bigg\{
	E_{\sF_t}\int_{t}^{\tilde{t}}
	f(s,X^{ t,X_{{t}}^{0,x;\bar{\theta}};\theta}_s,\theta_s)\,ds
	+E_{\sF_t}\int_{\tilde{t}}^T
	\left(
	f(s,X^{ t,X_{{t}}^{0,x;\bar{\theta}};\theta}_s,\theta_s)
	-f(s,X^{\tilde t,X_{\tilde{t}}^{0,x;\bar{\theta}};\theta}_s,\theta_s)
	\right)\,ds
	 \nonumber \\
&\quad\quad\quad
	+
	E_{\sF_t}\left[
	G\left(X^{ t,X_{{t}}^{0,x;\bar{\theta}};\theta}_T\right)
	-G\left(X_T^{\tilde t,X_{\tilde{t}}^{0,x;\bar{\theta}};\theta}\right)
	\right]
	\bigg\} \nonumber \\
&\geq
	-\esssup_{\theta\in\cU}\bigg\{
	L(\tilde t -t)
	+
	E_{\sF_t}
	\int_{\tilde t}^T L\left|   X^{ t,X_{{t}}^{0,x;\bar{\theta}};\theta}_s-  X^{\tilde t,X_{\tilde{t}}^{0,x;\bar{\theta}};\theta}_s \right| \,ds
	+E_{\sF_t} \bigg[
	L\left|   X^{ t,X_{{t}}^{0,x;\bar{\theta}};\theta}_T-  X^{\tilde t,X_{\tilde{t}}^{0,x;\bar{\theta}};\theta}_T \right|
	\bigg]
	\bigg\} \nonumber \\
&\geq
	-L(\tilde t -t)-\esssup_{\theta\in\cU}\bigg\{
	KTLE_{\sF_t}
	 \left|   X^{ t,X_{{t}}^{0,x;\bar{\theta}};\theta}_{\tilde t}-  X_{\tilde{t}}^{0,x;\bar{\theta}} \right|
	+LKE_{\sF_t} 
	\left|   X^{ t,X_{{t}}^{0,x;\bar{\theta}};\theta}_{\tilde t}-  X_{\tilde{t}}^{0,x;\bar{\theta}} \right|
	\bigg\} \nonumber \\
&=
	-L(\tilde t -t)-\esssup_{\theta\in\cU}\bigg\{
	KL(T+1)E_{\sF_t}
	 \left|   X^{ t,X_{{t}}^{0,x;\bar{\theta}};\theta}_{\tilde t}-  X_{\tilde{t}}^{0,x;\bar{\theta}} \right|
	\bigg\} \nonumber \\
&\geq
	-L(\tilde t -t)-(T+1)KL\esssup_{\theta\in\cU}\bigg\{
	E_{\sF_t}\left[
	 \left|   X^{ t,X_{{t}}^{0,x;\bar{\theta}};\theta}_{\tilde t}      - X^{ t,X_{{t}}^{0,x;\bar{\theta}};\theta}_{ t}\right|+     \left| X_{\tilde{t}}^{0,x;\bar{\theta}}-X_{{t}}^{0,x;\bar{\theta}} \right|\right] 
	\bigg\} \nonumber \\
&\geq
	-L(\tilde t -t)-    2(T+1)K^2L \left(1+\left|X_{{t}}^{0,x;\bar{\theta}} \right|\right)  (\tilde{t}-t)^{1/2}\longrightarrow 0,\quad \text{as }|t-\tilde{t}|\rightarrow 0, \label{est-prop-cont}
\end{align}
}
where we used the basic properties listed in Lemma \ref{lem-SDE}. Analogously, one proves the continuity of $EV(t,X^{0,x;\bar{\theta}}_t)$ in $t$, which by the regularity of supermartingale implies the right continuity of $V(t,X^{0,x;\bar{\theta}}_t)$. Furthermore, by the BSDE theory, $E_{\sF_t}V(\tilde{t},X^{0,x;\bar{\theta}}_{\tilde{t}})$ is continuous in $t\in[0,\tilde{t}]$, and this together with the above calculations implies the left continuity of $V(t,X^{0,x;\bar{\theta}}_t)$ in $t$. Hence, $\left\{V(s,X_s^{0,x;\bar{\theta}})\right\}_{s\in[0,T]}$ is a continuous process and we prove assertion (iii).

 For any $x,y\in\bR^d$ and any $\bar\theta\in\cU$, by definition of the value function, we have
\begin{align}
&|V(t,x)-V(t,y)|+|J(t,x;\bar\theta)-J(t,y;\bar\theta)|
\\
&
\leq 2 \esssup_{\theta\in\cU}E_{\sF_t}\bigg[\int_t^T\!\!|f(s,X_s^{t,x;\theta},\theta_s)-f(s,X_s^{t,y;\theta},\theta_s)|\,ds
+|G(X_T^{t,x;\theta})-G(X_T^{t,y;\theta})|\bigg]
\nonumber\\
&\leq
2 \esssup_{\theta\in\cU}\left\{ 
\int_t^T\!\! L E_{\sF_t} |X_s^{t,x;\theta}- X_s^{t,y;\theta} |\,ds
+LE_{\sF_t} |X_T^{t,x;\theta}-X_T^{t,y;\theta}|
\right\}
\nonumber\\
&\leq
2L(T-t)K|x-y|+2LK|x-y| \label{est-hder}
\end{align}
from which one derives the Lipschitz continuity of $V(t,x)$ and $J(t,x;\theta)$ in $x$. This yields assertion (iv).

Finally,  for any $(s,y),(t,x)\in[0,T]\times \bR^d$, we assume w.l.o.g. $s\geq t$. 

On the one hand, if $(s,y)$ tends to $(t,x)$, by assertions (iii) and (iv), we have
\begin{align*}
|V(t,x)-V(s,y)|
&\leq \left|V(t,x)-V\left(s,X_s^{t,x;\theta}\right)\right|
      + \left| V\left(s,X_s^{t,x;\theta}\right) -V(s,x)\right|
      +\left|V(s,x)-V(s,y)\right|\\
 &\leq
      \left|V(t,x)-V\left(s,X_s^{t,x;\theta}\right)\right|
      +L_V\left( \left| X_s^{t,x;\theta} -x \right|
      +|x-y|\right) \\
 &\rightarrow 0 \quad \text{a.s.}
\end{align*}
On the other hand, if $(t,x)$ tends to $(s,y)$, it holds that
\begin{align*}
|V(t,x)-V(s,y)|
&\leq
	\left| V(s,y)-E_{\sF_t}V(s,y)    \right|
	+\left| E_{\sF_t}[V(s,y)-V(s,x)]   \right|
	+\left| E_{\sF_t}[V(s,x)-V(s,X^{t,x}_s)]   \right|
\\
&\quad
	+\left| E_{\sF_t}V(s,X^{t,x}_s)  -V(t,x)\right|\\
&\leq
	\left| V(s,y)-E_{\sF_t}V(s,y)    \right|
	+C|x-y|
	+CK(1+|x|)(s-t)^{1/2}
\\
&\quad
	+C\left[s-t+(1+|x|)(s-t)^{1/2}  \right] \quad \text{(by estimate \ref{est-prop-cont})}\\
&\rightarrow 0 \quad \text{a.s.}
\end{align*}
Hence, $V(t,x)$ is continuous on $[0,T]\times\bR^d$ for almost all $\omega\in\Omega$, and by $(\cA1)$, it is obvious that $\esssup_{\omega\in\Omega}\sup_{(t,x)\in[0,T]\times\bR^d} |V(t,x)|\leq LT+L$. For $J(t,x)$, it follows similarly.  
\end{proof}

\end{appendix}

\bibliographystyle{siam}

\end{document}